\documentclass[11pt]{amsart}

\usepackage{mathrsfs}  
\usepackage[backref,colorlinks]{hyperref} 
\usepackage{amsmath,amssymb,amsthm,bm}
\usepackage[foot]{amsaddr}
\usepackage[abbrev,msc-links]{amsrefs}   
\usepackage[utf8]{inputenc}
\usepackage{tikz}
\usepackage{enumerate}
\usepackage{enumitem}
\usepackage{graphicx}
\usepackage{lineno}
\usepackage{mathtools}

\DeclarePairedDelimiter\ceil{\lceil}{\rceil}
\DeclarePairedDelimiter\floor{\lfloor}{\rfloor}

\DeclarePairedDelimiter{\final}{\langle}{\rangle}

\DeclareMathOperator{\dist}{dist}

\theoremstyle{plain}
\newtheorem{thm}{Theorem}[section]

\newtheorem{prop}[thm]{Proposition}
\newtheorem{clm}[thm]{Claim}

\newtheorem{lem}[thm]{Lemma}

\newtheorem*{clm*}{Claim}

\theoremstyle{definition}
\newtheorem{rem}[thm]{Remark}
\newtheorem{dfn}[thm]{Definition}

\newtheorem{obs}[thm]{Observation}

\numberwithin{equation}{section}

\normalsize                              
\def\NN{\mathbb N}
\def\ZZ{\mathbb Z}

\title{Slow graph bootstrap percolation  I:  Cycles}
 \author{David Fabian $^{\ast}$}
 \author{Patrick Morris $^{1,\dagger}$}
 \author{Tibor Szab\'o $^{2,\ddagger}$}

 \address{$^1$ Departament de Matem\`atiques, Universitat Polit\`ecnica de Catalunya (UPC), Barcelona, Spain.}
 \address{$^2$ Institute of Mathematics, Freie Universit\"at Berlin, Germany}

 \thanks{$^\ast$ Research supported by the Deutsche Forschungsgemeinschaft (DFG)
 Graduiertenkolleg “Facets of Complexity” (GRK 2434). }

 \thanks{$^\dagger$ Research supported by   the DFG Walter Benjamin program - project number 504502205.}

 \thanks{$^\ddagger$ Research
 supported by the DFG under Germany’s Excellence Strategy - The Berlin Mathematics Research Center
 MATH+ (EXC-2046/1, project ID: 390685689).}

 \email{dfabian@mailbox.org, pmorrismaths@gmail.com,  szabo@math.fu-berlin.de}

\date{\today}

\begin{document}

\begin{abstract}
Given a fixed graph $H$ and an $n$-vertex graph $G$,  the $H$\emph{-bootstrap percolation process} on $G$ is defined to be the sequence of graphs $G_i$, $i\geq 0$ which starts with $G_0 := G$ and in which $G_{i+1}$ is obtained from $G_i$ by adding every edge that completes a copy of $H$. We are interested in $M_H(n)$ which is  the maximum number of steps, over all $n$-vertex graphs $G$, that this process takes to stabilise. 
We determine this maximum  running time precisely when $H$ is a cycle, giving the first infinite family of graphs $H$ for which an exact solution is known. We find that $M_{C_k}(n)$ is of order $\log_{k-1}(n)$ for all $3\leq k\in \mathbb{N}$. Interestingly though, the function exhibits different behaviour depending on the parity of $k$ and the exact location of the values of $n$ for which $M_H(n)$ increases is determined by the Frobenius number of a certain numerical semigroup depending on $k$.
\end{abstract}

\maketitle

\section{Introduction}
Introduced by Bollob\'as \cite{bollobas1968weakly} in 1968, the $H$\textit{-bootstrap percolation process} ($H$\emph{-process} for short) on an $n$-vertex graph $G$ is the sequence $(G_i)_{i\geq 0}$ of graphs defined by $G_0 := G$ and
    \begin{linenomath} \begin{align*}
    V(G_i) &:= V(G), \\
    E(G_i) &:= E(G_{i-1}) \cup \left\{e\in\binom {V(G)} 2 : n_H\left(G_{i-1}+e\right)>n_H(G_{i-1})\right\}, 
    \end{align*} \end{linenomath}
for $i\geq 1$,  where  $n_H(G)$ denotes the number of copies of $H$ in $G$.
We call $G$ the \emph{starting graph} of the process.
Note that this process will \emph{stabilise} at some point with $\tilde G=G_t=G_{t+1}=\cdots$ for some $n$-vertex $\tilde{G}$ and $t\in \NN$. We call $\tilde G$ the \emph{final graph} of the process and if $\tilde G=K_n$,  we say the process \emph{percolates}. The $H$-bootstrap process is linked to the  notion to the study of \emph{cellular automata}, a deep topic introduced by von Neumann (see~\cite{neumann1966theory}).  Indeed, a common setup of a cellular automaton is to study the spread of a virus through a (hyper-)graph where some vertices are initially infected and the virus is passed on to other vertices at each time step according to some local homogeneous rule. By considering the hypergraph whose vertex set is $E(K_n)$ and whose edge set encodes copies of $H$, one can view the $H$-bootstrap process as a cellular automaton. The literature on bootstrap percolation is vast and the topic has been studied from many different perspectives as the concept can describe important processes occurring in physics, sociology and computer science (see for example~\cite{adler2003bootstrap}). 

In recent years, this study has become prominent in extremal and probabilistic combinatorics with techniques from these areas being successfully applied to address key problems from other areas  (see for example the very nice survey of  Morris~\cite{morris2017bootstrap}) as well as  new lines of research in combinatorics being motivated from this connection. In particular, inspired by analogous questions for similar automata studied in physics~\cite{chalupa1979bootstrap}, Balogh, Bollob\'as and Morris~\cite{balogh2012graph} initiated the study of the $H$-bootstrap process (and coined this terminology) when the starting graph $G$ is the Erd\H{o}s-Renyi random graph $G_{n,p}$ and asked for the threshold probability at which the process with initial graph $G_{n,p}$ percolates.

\subsection{The running time of bootstrap processes}
Most of the research on the graph bootstrap percolation process has focused on whether or not the process \emph{percolates}. Adopting the cellular automata view of a virus spreading, this translates to asking whether or not the virus will reach the whole population, which is certainly a natural  line of  investigation. Here,  we will rather be interested in \emph{how long} the virus will spread for, a question which one could also  imagine being  important in applications. This perspective, however, has been  considerably less explored until recently.

We define the \textit{running time} of the $H$-bootstrap process $(G_i)_{i\geq 0}$ with initial graph $G$ to be $\tau_H(G):=\min \{t \in \NN: G_t=G_{t+1}\}$, the time at which the process stabilises.  Bollob\'as posed the  extremal question of determining the \emph{maximum} running time of the $H$-bootstrap process. 
	
    \begin{dfn} \label{def:maxruntime}
    For $n\in\NN$, we define $M_H(n)$ to be the maximum running time of the $H$-bootstrap process over all choices of starting graph $G$ with $n$ vertices, that is,
        \[M_H(n):=  \max_{|{V(G)}|=n}\tau_H(G).\]
    \end{dfn}
	
The initial focus of research into maximum running times has been the case when $H$ is a clique. When $H=K_3$ and $G$ is a path with $n$ vertices, one can see that $\tau_H(G)=\ceil{\log_2(n-1)}$ as the distance between any pair of non-adjacent vertices approximately halves at each step. Moreover,  this happens for any pair of vertices in each connected component of \emph{any} initial graph. As the $n$-vertex path maximises the diameter of an $n$-vertex graph, we have $\tau_H(G)\le \ceil{\log_2(n-1)}$ for all $n$-vertex $G$ and hence  $M_H(n)=\ceil{\log_2(n-1)}$. For $K_4$, the maximum running time is much larger. Indeed, Bollob\'as, Przykucki, Riordan and Sahasrabudhe~\cite{bollobas2017maximum} and, independently, Matzke~\cite{matzke2015saturation} showed that $M_{K_4}(n)=n-3$, for all $n\ge 3$.

Bollob\'as, Przykucki, Riordan and Sahasrabudhe~\cite{bollobas2017maximum} also realised that the running times could be even longer for $K_r$-processes as $r$ grows and they gave constructions showing that $M_{K_r}(n)\ge n^{2-\lambda_r-o(1)}$ for $r\ge 5$, where $\lambda_r$ is some explicit constant such that $\lambda_r\rightarrow 0$ as $r\rightarrow \infty$.  However the same authors believed that there was a limit to how long the   $K_r$-bootstrap process could last, conjecturing that for all $r\ge 5$, $M_{K_r}(n)=o(n^2)$. It turns out that this conjecture was in fact false. Indeed, Balogh, Kronenberg, Pokrovskiy and Szab\'o~\cite{balogh2019maximum}  proved  that $M_{K_r}(n)=\Omega(n^2)$ for all $r\ge 6$.

Interestingly, their construction could not be pushed to give quadratic time for $K_5$, but they could show that $M_{K_5}(n)\ge n^{2-o(1)}$ using a connection to additive combinatorics and a  famous construction of Behrend~\cite{behrend1946sets} giving large sets free of $3$-term arithmetic progressions. 
Determining the asymptotics of $M_{K_5}(n)$, and in particular, if it can be quadratic or not, remains a very interesting open problem. 

Recently, Noel and Ranganathan \cite{noel_running_2022}, Hartarsky and Lichev \cite{hartarsky_maximal_2022}, and Espuny Díaz, Janzer, Kronenberg and Lada \cite{espuny_diaz_long_2022} extended the study of $M_H(n)$ to hypergraphs, their main focus being the case in which $H$ is a clique.

\subsection{The cycle bootstrap percolation process}
In a series of papers, we will initiate a systematic study of the maximum running time of $H$-bootstrap percolation processes as $H$ varies. Whilst our focus will predominantly be to study the asymptotics of the function $M_H(n)$ and its dependence on different properties of $H$, we begin our explorations by studying one family of graphs $H$ in detail, namely cycles. 
In this paper we determine the precise value of $M_{C_k}(n)$ for all $k\geq 3$ and all $n$ sufficiently large, providing the first infinite family of graphs $H$ for which an exact running time $M_H(n)$ is known. In fact before this work, the only functions $M_H(n)$ that have been fully determined for non-trivial choices of $H$ are when $H=K_3$, $H=K_4$ \cite{matzke2015saturation,przykucki2012maximal} as discussed above and the $3$-uniform clique on $4$ vertices minus an edge in the hypergraph setting \cite{espuny_diaz_long_2022}. 

    \begin{thm}\label{thm:cycles}
    Let $k\geq 3$. For sufficiently large $n\in\NN$ we have
        \begin{linenomath} \begin{equation}\label{eq:cycles_main}
        M_{C_k}(n) = \begin{cases}
        \ceil*{\log_{k-1} (n+k^2-4k+2)} & \text{if $k$ is odd} ;\\
        \ceil*{\log_{k-1}\left(2n+k^2-5k\right)} & \text{if $k$ is even}.
        \end{cases}
        \end{equation} \end{linenomath} 
    \end{thm}

    \begin{rem}
    In both the odd and the even case \eqref{eq:cycles_main} only holds when $n$ is larger than roughly $k^{\binom k2}$ (with some work this can be improved to $k^{k/2}$).
    For smaller $n$ the behaviour is different as a single $k$-cycle with a well-placed chord achieves a longer running time than the general constructions.
    \end{rem}

As discussed above, the bootstrap process of the triangle on an $n$-vertex graph has a running time of at most $\lceil\log_2 (n-1)\rceil$. The key observation is that in the $K_3$-bootstrap process the distance between two vertices is roughly halved in every step. With this in mind it should not be surprising that the maximum running time of the bootstrap process of $C_k$ for $k\geq 3$ is asymptotically of the order  $\log_{k-1} n$.  What is perhaps unexpected is  that the jumps of the monotone increasing function $M_{C_k}(n)$, i.e. those $n\in\NN$ with $M_{C_k}(n+1) =  M_{C_k}(n)+1$, behave differently in terms of $k$ depending on the parity of $k$. For odd $k$ the jumps are close to powers of $k-1$ while for even $k$ the jumps are near one half times powers of $k-1$. Moreover, the precise location of the jumps is determined by a quadratic function of $k$ in both the even and odd cases. This function turns out to be controlled by the  \textit{Frobenius number} of the numerical semi-group generated by $k-2$ and $k$, that is, the largest  natural number that cannot be expressed as an integral linear combination of $k-2$ and $k$ with non-negative coefficients. This link to an arithmetic problem of  determining the largest gap in a numerical semi-group presents itself naturally in the analysis of the $C_k$-process, as the last edge to be added before stabilising will correspond to this gap. 

As shown in \cite{FMSz2}, examples of $H$ such that $M_H(n)$ is asymptotically sublinear must have strong restrictions on their degree sequence. Indeed, it is shown there that if  every component of a graph $H$ has minimum degree at least 2 and maximum degree at least 3, then $M_H(n)=\Omega(n)$. In the case where both the maximum and  minimum degrees are 2, that is, when $H$ is a disjoint union of cycles, our second result shows that the maximum running time is controlled by the largest cycle in $H$. 
    
    \begin{thm}\label{thm:multiple_cycles} 
    If  $s\geq 2$ and $H := C_{k_1} \sqcup \ldots \sqcup C_{k_s}$ is the disjoint union of cycles of lengths $k_1 \geq \ldots \geq k_s$, then for sufficiently large $n$ we have that
        \[
        \log_{k_1-1}(n) - 1 \leq M_H(n) < \log_{k_1-1}(n) + 6s^6k_1^2.
        \]
    \end{thm}
	
\subsection*{Organisation.}
Necessary notation is given  in Section \ref{sec:notation} and some basic results on the $C_k$-process are collected in Section \ref{sec:cycle process}.
We prove Theorem \ref{thm:cycles} in Sections  \ref{sec:cycles_lower} and  \ref{sec:cycles_upper} after giving an outline of the proof in Section \ref{sec:cycles_outline}. Finally, in Section \ref{sec:multiple}, we establish Theorem \ref{thm:multiple_cycles}.

\section{Notation and Preliminaries}\label{sec:notation}
If $H$ is a graph and $v$ is a vertex of $H$ then $H-v$ denotes the graph obtained from $H$ by removing $v$ and all edges incident to it, i.e.
	\[
	V(H-v) = V(H)\setminus\{v\}, \qquad\qquad E(H-v) = E(H)\setminus\{ e\in E(H) : v\in e \}.
	\]
For an edge $e\in E(H)$ we define $H-e$ as the graph obtained by removing $e$ from the edge set. For an edge $e\in E(K_{v(H)})$ on $V(H)$, we define $H+e$ to be the graph with vertex $V(H+e)=V(H)$ and edge set $E(H+e)=E(H)\cup \{e\}$. 

If $G\subseteq G'$ and $X,Y$ are disjoint subsets of $V(G')$ we write
    \[
    E_G(X,Y) = \left\lbrace  xy \in E(G) : x\in X, y \in Y \right\rbrace.
    \]
We write $N_G(v)$ for the set of neighbours of $v$ in $G$. 

\subsection*{Paths} 
We denote the path on $n$ vertices by $P_n$, i.e.
    \[
    V(P_n) = \{ 0,\ldots,n-1 \} \qquad, \qquad E(P_n) = \left\lbrace \{ i,i+1 \} : 0\leq i < n-1 \right\rbrace.
    \]
The \emph{length} of a path is its number of edges.

\subsection*{Frobenius numbers}
The \emph{Frobenius number} $F(x,y)$ of two positive, coprime integers $x$, $y$ is the largest natural number that cannot be expressed as an integral linear combination of $x$ and $y$ with non-negative coefficients, i.e. 
    \[
    F(x,y) := \max \left( \ZZ\setminus\left\lbrace \alpha x + \beta y : \alpha,\beta \in \NN_0\right\rbrace \right)
    \]
where $\NN_0$ denotes the set of non-negative integers.
The precise formula $F(x,y) = xy - x - y$ is well-known.
A thorough treatise of Frobenius numbers and their generalisations can be found in \cite{alfonsin2005diophantine}.
We are interested in $F(k-2,k)$ for odd integers $k\geq 3$, in which case the above formula gives
    \begin{equation}\label{eq:Frobenius}
    F(k-2,k) = k^2-4k+2.
    \end{equation}

If $k$ is even we set $F'(k-2,k)$ to be the largest multiple of $\gcd(k-2,k) =2$ that cannot be written as an integral linear combination of $k-2$ and $k$ with non-negative coefficients, i.e.
    \begin{equation}\label{eq:FrobeniusDash}
    F'(k-2,k) := 2\cdot F\left(\frac{k-2}2,\frac k 2\right) = \frac{k^2}{2}-3k+2.
    \end{equation}

\subsection*{Sumsets}
Given $h\in\NN$ and a set $A$ of integers, $hA$ denotes the $h$-fold sumset
	\[ hA := \{ a_1+\ldots+a_h : a_1,\ldots,a_h\in A \}. \]

\subsection*{Graph bootstrap processes} 
Whenever the process $(G_i)_{i\geq 0}$ is clear from context, we say that a property of a graph holds \emph{at time} $i$ if $G_i$ has that property.

\subsection*{Stable graphs}
All graphs of an $H$-process on a given $n$-vertex graph $G$ will be considered as subgraphs of $K_n$.
We say that a graph $G$ is $H$\emph{-stable} if $n_H(G+e) = n_H(G)$ for every $e\in \binom{V(G)}2$.
For any graph $G$ we define $\langle G \rangle_H$ to be the final graph of the $H$-process on $G$.  A short induction shows that every $H$-stable graph containing $G$ must also contain every graph of the $H$-process on $G$. Therefore $\langle G \rangle_H$ is the smallest $H$-stable graph in which $G$ appears as a subgraph.

\subsection*{Graph homomorphisms}
Given  graphs $G,G'$, a map  $\phi:V(G)\rightarrow V(G')$ is a  \emph{graph homomorphism} if for any $e=uv\in E(G)$, we have that $\phi(e)=\phi(u)\phi(v)\in E(G')$. In order to signify the added condition that edges should map to edges we will write graph homomorphisms as $\phi:G\rightarrow G'$. We say that a graph homomorphism $\phi$ is injective if $\phi$ is injective on $V(G)$ and we say $\phi:G\rightarrow G'$ is a graph \emph{automorphism} if $G=G'$ and $\phi$ is injective (and hence bijective). We let $\mathrm{Hom}(G, G')$ denote the set of homomorphisms from $G$ to $G'$ and let $\mathrm{Aut}(G)$ denote the set of automorphisms of $G$. The following observation shows that injective graph homomorphisms are preserved through the graph bootstrap process. 

    \begin{obs}\label{obs:hom}
    Let $\varphi: G \to G'$ be an injective graph homomorphism, and let $(G_i)_{i\geq0}$, $(G'_i)_{i\geq0}$ be the respective $H$-processes on $G$ and $G'$.
    Then $\varphi \in \mathrm{Hom}(G_i, G'_i)$ for every $i\geq 0$.
    \end{obs}

    \begin{proof}
    The claim holds for $i=0$ because $G_0 = G$, $G'_0 = G'$.
    Let $i\geq 0$ and suppose that $\varphi \in \mathrm{Hom}(G_{i-1}, G'_{i-1})$.
    Let $e\in E(G_i)\setminus E(G_{i-1})$.
    Therefore  there exists a copy $H_i \subseteq G_i$ of $H$ in $G_i$ such that  $H_i - e \subseteq G_{i-1}$.
    We have $\varphi(H_i) - \varphi(e) = \varphi(H_i-e)$ because $\varphi$ is injective, and $\varphi(H_i-e) \subseteq G'_{i-1}$ since $\varphi \in \mathrm{Hom}(G_{i-1}, G'_{i-1})$.
    Thus, $\varphi(e) \in E(G'_i)$ by definition of the $H$-process on $G'$.
    \end{proof}

Two immediate consequences of Observation \ref{obs:hom} are that $G \subseteq G'$ implies $G_i \subseteq G'_i$ for all $i\geq 0$, and that $\mathrm{Aut}(G_i) \subseteq \mathrm{Aut}(G_{i+1})$,  for all $i\geq 0$.

\section{The cycle process} \label{sec:cycle process}
In this section, we collect some simple observations and lemmas about the $C_k$-process. 
Firstly, we  reduce the running times on disconnected $G$ to running times on connected starting graphs.

	\begin{obs}\label{obs:connectivity}
    Let $G$ be a graph with connected components $G^{(1)},\ldots,G^{(s)}$, and let $(G_i)_{i\geq 0}$ be its $C_k$-process.
    Then $G_i= G^{(1)}_i \cup\ldots\cup G^{(s)}_i$, and hence
        \[ 
        \final{G}_{C_k} = \final{G^{(1)}}_{C_k} \cup\ldots\cup \final{G^{(s)}}_{C_k}, \mbox{ and }    
        \tau_{C_k}(G) = \max \left\lbrace \tau_{C_k}(G^{(1)}) ,\ldots, \tau_{C_k}(G^{(s)}) \right\rbrace.
        \]
    \end{obs}
    
    \begin{proof}
    Suppose that at some step in the process the number of components decreases.
    Take the smallest $i$ for which there exists an edge $e\in E(G_i)$ whose endpoints lie in distinct components of $G$.
    At time $i-1$ there must be path of length $k-1$ between the endpoints of $e$, a contradiction.
    \end{proof}
    
Any component with less than $k$ vertices is $C_k$-stable and thus does not affect the process.
Therefore,
    \begin{linenomath} \begin{equation}\label{eq:connectivity}
    M_{C_k}(n) = \max \{ \tau_{C_k}(G) : G \text{ connected}, k\leq v(G) \leq n \}.
    \end{equation} \end{linenomath} 
For even $k$ another graph property that is preserved throughout the process is bipartiteness.

    \begin{lem}\label{lem:cycles_bipartiteness}
    Let $4\leq k\in 2\mathbb{N}$. If $G$ is a bipartite graph with partite sets $X,Y\subseteq V(G)$, so is $\final{G}_{C_k}$.
    \end{lem}
    
    \begin{proof}
    Let $(G_i)_{i\geq 0}$ be the $C_k$-process on $G$, and suppose for a contradiction that the final graph was not bipartite.
    Pick the smallest $i$ for which $G_i$ contains an edge $e$ whose endpoints lie in the same part.
    Then there exists a path of length $k-1$ between the endpoints of $e$ at time $i-1$, a contradiction as $k-1$ is odd.
    \end{proof}

Next we  show that if a graph $G$ is not $C_k$-stable, then cycles will quickly appear everywhere.  
    
    \begin{lem}\label{lem:cycles_cycleattimetwo}
    Let $ G$ be a connected graph with $\tau_{C_k}( G) \geq 1$. Then in the $C_k$-process on $ G$ every vertex is contained in a $k$-cycle at time $2$.
    \end{lem}

    \begin{proof}
    Let $( G_i)_{i\geq 0}$ be the $C_k$-process on $ G$. 
    Since $\tau_{C_k}( G) \neq 0$, there exists a $k$-cycle $C$ in $ G_1$.
    Let $x \in V( G)\setminus V(C)$, and let $Q$ be a shortest path from $x$ to $V(C)$ in $ G_1$.
    If $Q$ has length at least $k-1$ the first $k$ vertices of $Q$ starting from $x$ form a path of length $k-1$ with endpoint $x$. 
    If the length of $Q$ is smaller than $k-1$ we can extend $Q$ to a path of length $k-1$ using vertices along $C$.
    In either case, the vertices of this path of length $k-1$, one of which is $x$, form a $k$-cycle in $ G_2$.
    \end{proof}

Our next result treats the case that the starting graph is complete bipartite with an extra edge. 

    \begin{lem}\label{lem:cycles_completebipartitetoclique}
    Let $k\geq 3$, and let $z,z' \in V(K_{\floor{k/2},\ceil{k/2}})$ be vertices from the same partite set of $K_{\floor{k/2},\ceil{k/2}}$.
    Then $\tau_{C_k}(K_{\floor{k/2},\ceil{k/2}} + \{ zz' \}) \leq 2$ and $\final{K_{\floor{k/2},\ceil{k/2}}+\{ zz' \}}_{C_k} = K_k$.
    \end{lem}

    \begin{proof}
    Let $ G := K_{\floor{k/2},\ceil{k/2}} + \{ zz' \}$ and denote the partite sets of $K_{\floor{k/2},\ceil{k/2}}$ by $X$ and $Y$ such that $|X| = \ceil{k/2}$ and $|Y| = \floor{k/2}$.
    If $k$ is odd, then for any two distinct $x,x'\in X$ we can find a Hamilton path, which has length $k-1$, from $x$ to $x'$ in $K_{\floor{k/2},\ceil{k/2}}$.
    Thus $X$ is a clique after one step in the $C_k$-process on $ G$.
    At time $1$, $X \setminus \{ x \}$ and $Y\cup \{ x \}$ are partite sets of a complete bipartite graph of size $\floor{k/2}$ and $\ceil{k/2}$, respectively.
    Therefore $Y \cup \{ x \}$ is a clique at time $2$.
    This shows the claim for odd $k$.
    Now assume that $k$ is even, in particular, $k \geq 4$ so both $|X|\geq 2$ and $|Y|\geq 2$.
    Since $|X| = |Y|$ we may further assume that $z,z'\in X$.
    For any distinct $y,y'\in Y$ we can pick a Hamilton path from $y$ to $z$ in the complete bipartite   graph $ G-y'-z'$ and extend that path to a $yy'$-path of length $k-1$ in $ G$ by $zz'$ and $z'y'$.
    Then $Y$ must be a clique at time $1$.
    Analogous arguments show that $X$ is a clique after one more step and hence the claim follows.
    \end{proof}

Finally we categorise the possible final graphs for connected starting graphs that are not $C_k$-stable.

    \begin{lem}\label{lem:cycles_smalldistance}
    Let $ G$ be a connected graph of order at least $k+1$ which contains a copy of $C_k$. The final graph $\final{ G}_{C_k}$ is a clique if $k$ is odd or $ G$ is non-bipartite, and a complete bipartite graph if $k$ is even and $ G$ is bipartite.
    \end{lem}

    \begin{proof}
    In $\final{ G}_{C_k}$ the endpoints of any path of length $k-1$ are adjacent.
    Therefore the shortest path between any two vertices has length less than $k-1$.
    Choose vertices $v_j$, $j\in [0,k-1]$, in $ G$ that form a $k$-cycle $C$ with edges $v_jv_{j+1}$. Here and for the rest of this proof addition and subtraction in the subscript are always performed modulo $k-1$.
    Every $x\in V( G)\setminus\{ v_0,\ldots,v_{k-1} \}$ has a neighbour on $C$ in $\final{ G}_{C_k}$  because a shortest path from $x$ to $C$ in $ G$ can always be extended to a path of length $k-1$ by vertices of $C$.
    If $xv_j \in E(\final{ G}_{C_k})$ then $xv_jv_{j-1}\ldots v_{j+2}$ is a path of length $k-1$ so $xv_{j+2} \in E(\final{ G}_{C_k})$.
    In the case that $k$ is odd the above implies that every vertex of $C$ is adjacent to every other vertex of $ G$ in $\final{ G}_{C_k}$.
    Thus for any two distinct vertices $x,y\in V(G)$ we can find a $k$-cycle containing $x$ but not $y$.
    Indeed if both $x$ and $y$ lie on $C$ we can replace $y$ by a vertex not on $C$ to obtain the desired cycle.
    In the case that neither $x$ nor $y$ is a vertex of $C$ we can replace an arbitrary vertex of $C$ by $x$.
    Repeating the above argument for such a cycle gives $xy \in E(\final{ G}_{C_k})$.
    
    Now assume that $k$ is even.
    Let 
        \[
        X := \{ v_j : j\equiv 0 \mod 2 \}  \qquad,\qquad Y :=  \{ v_j : j\equiv 1 \mod 2 \}.
        \]
    Then every vertex outside $C$ is adjacent in $\final{{G}}_{C_k}$ to all vertices in $X$ or all vertices in $Y$.    Define
        \[
        X' := \left\lbrace z \in V( G)\setminus V(C) : Y \subseteq N_{\final{ G}_{C_k}}(z) \right\rbrace 
        \qquad,\qquad 
        Y' := \left\lbrace z \in V( G)\setminus V(C) : X \subseteq N_{\final{ G}_{C_k}}(z) \right\rbrace .
        \]
    One of these two sets, say $X'$, must be non-empty.
    For any $x\in X'$, $y\in Y'$,     $yv_0v_1\ldots v_{k-3} x$ is an $xy$-path of length $k-1$ in $\final{ G}_{C_k}$.
    Furthermore for any $j, j' \in [0,k-1]$, with $v_j\in X$, $v_{j'}\in Y\setminus\{v_{j-1},v_{j+1}\}$ and any $x\in X'$
        \[
        v_{j'}v_{j'+1}\ldots v_{j-1} x  v_{j'-2}\ldots v_{j+1}v_j
        \]
    is an $v_jv_{j'}$-path of length $k-1$.
    Therefore $\final{ G}_{C_k}$ contains a complete bipartite graph whose partite sets are $X\cup X'$ and $Y\cup Y'$.
    If $ G$ is bipartite we are done by Lemma \ref{lem:cycles_bipartiteness}.
    Otherwise the claim follows from Lemma \ref{lem:cycles_completebipartitetoclique}.
    \end{proof}

\section{Proof outline}\label{sec:cycles_outline}
We begin by discussing the leading term $\log_{k-1}(n)$ in both the even and odd case of Theorem \ref{thm:cycles}. This logarithmic behaviour of $M_{C_k}(n)$ is a consequence of the  following key lemma  which essentially shows a decrease of the diameter by a factor of $k-1$ in each step of the $C_k$-process. 

    \begin{lem}\label{lem:cycles_distance}
    Let $(G_i)_{i\geq0}$ be the $C_k$-process on a connected graph $G=G_0$, and let $x,y\in V(G)$. 
    For each $i\geq 1$, the distance $\dist_{G_i}(x,y)$ satisfies
        \begin{equation} \label{eq:distance decrease}
        \frac{\dist_{G_0}(x,y)}{(k-1)^i}\leq  \dist_{G_i}(x,y) \leq \floor*{\frac{\dist_{G_{0}}(x,y)}{(k-1)^i}} + k-2.         
        \end{equation}
    When $\dist_{G_0}(x,y)$ is a multiple of $(k-1)^i$ the upper bound can be improved to
        \begin{linenomath} \begin{equation}\label{eq:gen_improved}
        \dist_{G_i}(x,y) \leq \frac{\dist_{G_{0}}(x,y)}{(k-1)^i}.
        \end{equation} \end{linenomath} 
    \end{lem}
    
    \begin{proof}
    Observe that for any edge $e\in E(G_i)\setminus E(G_{i-1})$ one can find a path of length $k-1$ between its endpoints in $G_{i-1}$.
    Given a shortest $xy$-path in $G_i$, replacing every edge on the path which is not present at time $i-1$ by a suitable path of length $k-1$ yields an $xy$-walk of length at most $(k-1)\cdot\dist_{G_i}(x,y)$ in $G_{i-1}$.
    From this we deduce that $\dist_{G_{i-1}}(x,y) \leq (k-1) \dist_{G_i}(x,y)$ and thus $\dist_{G_0}(x,y) \leq (k-1)^i \dist_{G_i}(x,y)$, which gives the lower bound in \eqref{eq:distance decrease}. 
    To obtain the upper bound on $\dist_{G_i}(x,y)$, write $\dist_{G_{i-1}}(x,y) = q\cdot (k-1) +r$ for suitable $q,r\in\NN_0$, $0\leq r \leq k-2$, and choose a path $u_0\ldots u_{q(k-1)+r}$ from $x$ to $y$ in $G_{i-1}$. In $G_i$, $u_0u_{k-1}\ldots u_{q(k-1)}u_{q(k-1)+1}\ldots u_{q(k-1)+r}$ is a path of length $q+r$ from $x$ to $y$.
    Since $r\leq k-2$,  we obtain
        \begin{linenomath} \begin{equation}\label{eq:gen_dist}
        \dist_{G_i}(x,y) \leq q+r  \leq \frac{q\cdot(k-1)+r-(k-2)}{k-1} + k-2 =\frac{\dist_{G_{i-1}}(x,y)-(k-2)}{k-1} + k-2.
        \end{equation} \end{linenomath} 
    We can bound the left hand side by just $q$ whenever $\dist_{G_i}(x,y)$ is divisible by $k-1$.
    An inductive application of \eqref{eq:gen_dist} yields the claim. 
    \end{proof}

Lemma \ref{lem:cycles_distance} (as well as some of the simple results from Section \ref{sec:cycle process}) already suffices to establish that $ \log_{k-1}(n)-1\leq M_{C_k}(n)\leq \log_{k-1}(n)+c_k$ for some $c_k>0$. Indeed, for the lower bound, one can consider the path $P_n$ as the starting graph and use Lemma \ref{lem:cycles_smalldistance} and the lower bound in \eqref{eq:distance decrease} to exhibit an edge that is eventually added but not until at least $\log_{k-1}(n)-1$ steps have passed.  For the upper bound, one can use 
\eqref{eq:connectivity} and the upper bound in \eqref{eq:distance decrease} to reduce to the situation where all distances in the graph $G_i$ are at most $k-1$ after $\log_{k-1}(n)$ steps. Then using Lemma \ref{lem:cycles_smalldistance} on appropriate subgraphs of $G_i$ will show that after another $c_k$ steps, the process will terminate, for some constant $c_k$ dependent only on $k$. 

The main contribution of this paper is getting the exact expressions for the maximum running times for both even and odd cycles. 
We split the proof of Theorem \ref{thm:cycles} into an upper bound part,  and a lower bound part. The upper bounds will be established by the following theorem. 

    \begin{thm}[Upper bound part]\label{thm:strategy_upperbound}
    Let $k\geq 3$, and let $G$ be a connected graph on at least $k+1$ and at most $n$ vertices with $C_k$-process $(G_i)_{i\geq 0}$ such that $\final{G}_{C_k} \neq G$.
    Define
        \begin{equation}\label{eq:defofr}
        r = r(n,k) := \begin{cases}
        \ceil*{\log_{k-1} (n+k^2-4k+2)} & \mbox{ if } k \mbox{ is odd}; \\
        \ceil*{\log_{k-1}\left(2n+k^2-5k\right)} & \mbox{ if } k \mbox{ is even}.
        \end{cases}
        \end{equation}
    If $n$ is sufficiently large the following hold:
        \begin{enumerate}[label=(\roman*)]
        \item \label{upperpart1} If $k$ is odd, then $xy \in E(G_r)$ for every distinct $x,y\in V(G)$.
        \item \label{upperpart2} If $k$ is even and $G$ bipartite with parts $X,Y\subseteq V(G)$ then $xy\in E(G_r)$ for any $x\in X$, $y\in Y$.
        \item \label{upperpart3} For even $k$ and non-bipartite $G$,  we have $xy \in E(G_r)$ for any distinct $x,y\in V(G)$.
        \end{enumerate}
    \end{thm}

We remark that by the definition of $r$, \eqref{eq:Frobenius} and \eqref{eq:FrobeniusDash}, $r$ is the unique natural number satisfying
    \begin{equation}\label{eq:rF}
    (k-1)^{r-1} - F(k-2,k) \ \leq \ n-1 \ < \  (k-1)^r - F(k-2,k),
    \end{equation}
when $k$ is odd.
Likewise
    \begin{equation}\label{eq:rF'}
    \frac{(k-1)^{r-1}-(k-1)}{2} - F'(k-2,k) + 2 \leq n < \frac{(k-1)^r-(k-1)}{2} - F'(k-2,k) + 2,
    \end{equation}
when $k$ is even. Theorem \ref{thm:strategy_upperbound} is proven in Section \ref{sec:cycles_upper} and follows from a delicate analysis. For parts \ref{upperpart1} and \ref{upperpart2}, we will identify an appropriate $xy$-path and deduce the result by considering cycle processes on paths. For \ref{upperpart3}, which is the most challenging, this will no longer work. Using subpaths  can in fact get close to the result, even showing $x$ and $y$ are adjacent after just $r+1$ steps. However, to get our exact result we need to take a more refined look at the exact edges that are added and at what time, in order to have $x$ and $y$ adjacent at time $r$. This leads us to consider walks instead of paths, which locally behave like paths with respect to the process.

To obtain a lower bound of the form $M_{C_k}(n) \geq r$ we need to specify a starting graph $G$ and an edge $e\in \binom{V(G)}2$ such that $e$ is present at time $r$ but not at time $r-1$.
In view of Theorem \ref{thm:strategy_upperbound} it suffices to give a pair of vertices (from different partite sets if $G$ is bipartite and $k$ is even) that are not adjacent at time $r-1$. This is the content of the following theorem.
    
    \begin{thm}[Lower bound part]\label{thm:strategy_lowerbound}
    Let $k\geq 3$, and let $G$ be a graph with $C_k$-process $(G_i)_{i\geq 0}$. Define $r$ as in \eqref{eq:defofr}, and set
        \begin{equation}\label{eq:defofell}
        \ell = \ell(n,k) := \frac{(k-1)^{r-1}-(k-1)}{2} - F'(k-2,k) - 1
        \end{equation}
    when $k$ is even.
    Then the following hold for $n$ sufficiently large:
        \begin{enumerate}
        \item \label{lowerpart1} If $k$ is odd and $G=P_n$, then $\{ 0, (k-1)^{r-1}-F(k-2,k) \} \notin E(G_i)$ for $i<r$.
        \item \label{lowerpart2} If $k$ is even and $G=P^\Delta$ (see Figure \ref{fig:pathtriangle}) on $\ell+3\leq n$ vertices then for the vertices $v_\ell,w_\ell\in V(P^\Delta)$ we have that $\{ v_\ell, w_{\ell}\} \notin E(G_i)$ for $i<r$.
        \end{enumerate}
    \end{thm}

    \begin{figure}[h]
    \centering
    \includegraphics[width=0.7\linewidth]{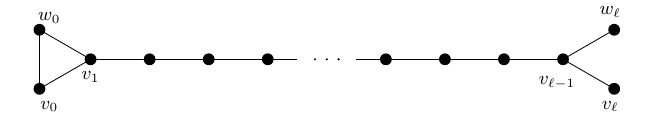}
    \caption{A visualisation of $P^\Delta$.}
    \label{fig:pathtriangle}
    \end{figure}

Theorem \ref{thm:strategy_lowerbound} is proven in Section \ref{sec:cycles_lower} via a careful analysis of when edges appear in the cycle process. As with the upper bounds, the odd case is easier and only requires an understanding of the process on paths. The even case is more involved with pairs of vertices in $P^\Delta$ behaving differently depending on their location in $P^\Delta$. In our analysis, we have to capture the idea of non-bipartiteness `spreading' from the triangle side of $P^\Delta$ along the path and track the time at which it reaches certain vertices. 

We finish this section by explicitly deducing Theorem \ref{thm:cycles} from Theorems \ref{thm:strategy_upperbound}  and \ref{thm:strategy_lowerbound}.

    \begin{proof}[Proof of Theorem \ref{thm:cycles}]
    We begin by proving the upper bound. We assume that $n$ is sufficiently large so that Theorem~\ref{thm:strategy_upperbound} holds and so that $r(n,k)\geq \binom{k}{2}$. Therefore we  have that any $k$-vertex graph $G$ stabilises in at most $r$ steps and so \eqref{eq:connectivity} tells us that we can restrict ourselves to connected starting graphs on at most $n$ and at least $k+1$ vertices.
    When $k$ is odd or the starting graph is non-bipartite, then the desired upper bound follows from parts \ref{upperpart1} and \ref{upperpart3} of Theorem \ref{thm:strategy_upperbound}, which state that by round $r$ our process reaches the complete graph, which is $C_k$-stable.
    If $k$ is even and the starting graph is bipartite with parts $X$ and $Y$, then part \ref{upperpart2} tells us that at time $r$ there is a complete bipartite graph between $X$ and $Y$, which by Lemma \ref{lem:cycles_bipartiteness} must be the final graph of the process.
    To obtain the lower bounds observe that $\ell \leq n-3$ for even $k$ by definition of $\ell$ and $r$ and that the edges specified in parts \eqref{lowerpart1} and \eqref{lowerpart2} of Theorem \ref{thm:strategy_lowerbound} are not present at time $r-1$, but will be added eventually by Theorem~\ref{thm:strategy_upperbound} (in fact in the next step). So the process is not finished after $r-1$ steps.
    \end{proof}

\section{Lower bounds}\label{sec:cycles_lower}
In this section we prove Theorem \ref{thm:strategy_lowerbound}, treating each part individually.
In both parts the following set will be convenient to get a handle on when an edge appears in the $C_k$-process on a path.

    \begin{linenomath} \begin{equation}\label{eq:defofAi}
    A_i := \left\lbrace (k-1)^i-\alpha\cdot(k-2)-\beta\cdot k : \alpha,\beta\in\NN_0 \right\rbrace
    \end{equation} \end{linenomath} 
    
Note that when $k$ is even, $A_i$ consists of odd numbers while for odd $k$ there is no restriction on the parity.
The $A_i$ form an increasing sequence with respect to inclusion because for any $\alpha, \beta\in \NN_0$,
    \[
    (k-1)^i - \alpha (k-2) - \beta k = (k-1)^{i+1} - (\alpha+(k-1)^i)\cdot (k-2) - \beta k \in A_{i+1}.
    \]

\subsection*{Proof of part \eqref{lowerpart1}}
Let $(P^i)_{i\geq 0}$ be the $C_k$-process on $P_n$ and recall that we chose $\{0,\ldots,n-1\}$ as the vertex set of $P_n$.

    \begin{lem}\label{lem:gen_fromright}
    If $xy\in E(P^i)$ for some $x,y \in V(P_n), i\geq 0$ then $y-x \in A_i$.
    \end{lem}
    
    \begin{proof}
    We prove the claim by induction on $i\geq 0$.

    $i=0$: By definition all edges in $P^0$ are of the form $\{ x, x+1 \}$.
    
    $i\geq 1$: Let $xy\in E(P^i)$. If $xy$ was already present at time $i-1$, the induction hypothesis and the inclusion $A_{i-1} \subseteq A_i$ give $y-x \in A_i$.
    Suppose $xy\notin E(P^{i-1})$. Let $v_0,\ldots,v_{k-1}$ be a path from $v_0 := x$ to $v_{k-1} := y$ in $P^{i-1}$. By the induction hypothesis there exist $\alpha_1,\ldots,\alpha_{k-1},\beta_1,\ldots,\beta_{k-1}$ such that  $v_j-v_{j-1} = (k-1)^{i-1}-\alpha_j\cdot(k-2) - \beta_j\cdot k$ for $j\in[k-1]$. Then
        \[
        y-x = \sum_{j=1}^{k-1} v_j-v_{j-1} = \sum_{j=1}^{k-1} \left( (k-1)^{i-1}-\alpha_j\cdot(k-2) - \beta_j\cdot k\right)  = (k-1)^i - \sum_{j=1}^{k-1} \alpha_j\cdot (k-2) -  \sum_{j=1}^{k-1} \beta_j\cdot k,
        \]
        completing the inductive step.
    \end{proof}
    
Lemma \ref{lem:gen_fromright} assures that whenever $d\in\NN$ is an integer that cannot be expressed as $d = \alpha (k-2) + \beta k$ for suitable $\alpha,\beta\in\NN_0$ then $(k-1)^i-d$ does not lie in $A_i$ and hence any edge $xy$ with $y-x = (k-1)^i-d$ cannot be present at time $i$.
Therefore the edge $\{ 0, (k-1)^{r-1}-F(k-2,k) \}$ is not present in $P^{r-1}$ by the definition of the Frobenius number $F(k-2,k)$ (see \eqref{eq:Frobenius}).
This shows   Theorem \ref{thm:strategy_lowerbound} part \eqref{lowerpart1}.

\subsection*{Proof of part \eqref{lowerpart2}}
To show part \eqref{lowerpart2} we recall the graph $P^\Delta$ defined by Figure \ref{fig:pathtriangle} and assume that $n$, and thus $\ell$, is sufficiently large so we do not run into degenerate cases, say $\ell\geq 3$. 
An important feature of the graph $P^\Delta$ for us is that it is a non-bipartite graph that maximises the length of a shortest odd walk between two vertices for fixed $n$.

Let $(P^{\Delta,i})_{i\geq 0}$ be the $C_k$-process on $P^\Delta$.
Recall that our goal is to show $v_\ell w_\ell \notin E(P^{\Delta,r-1})$.
To do so we will set up an analogue of Lemma \ref{lem:gen_fromright} for $P^{\Delta}$.
Call an edge $v_jv_{j'}$, $v_j w_{j'}$ or $w_jw_{j'}$ \emph{even} is $j-j'$ is even, and \emph{odd} if $j-j'$ is odd.
Lemma \ref{lem:cycles_analogueodd} below is the analogue of Lemma \ref{lem:gen_fromright} dealing with odd edges, while Lemma \ref{lem:lower_structure} deals with the even edges. Both rely on the following auxiliary statement and the assumption that $k$ is even:

    \begin{lem}\label{lem:cycles_endpointofevenedge}
    For every $i\geq 0$, the largest $j\in [\ell]$ such that $v_j$ or $w_j$ is an endpoint of an even edge in $P^{\Delta,i}$ is at most $(k-1)^i-1$.
    \end{lem}

    \begin{proof}
    The only even edge in $P^{\Delta,0}$ is $v_0w_0$ so the claim holds for $i=0$.
    Let $i\geq 1$ and suppose the claim holds for $i-1$.
    Since $(k-1)^i-1 > (k-1)^{i-1} -1$ it suffices to show that whenever $v_j$ or $w_j$ is the endpoint of an even edge in $E(P^{\Delta,i})\setminus E(P^{\Delta,i-1})$ one has $j \leq (k-1)^i-1$.
    Let $u_{j_0}u_{j_{k-1}}$ be an even edge in $E(P^{\Delta,i})\setminus E(P^{\Delta,i-1})$ and let $u_{j_0}\ldots u_{j_{k-1}}$ be a path in $P^{\Delta,i-1}$ such that $u_{j_t} \in \{ v_{j_t},w_{j_t} \}$ for $0 \leq t \leq k-1$.
    For parity reasons there exists at least one even edge on that path.
    Let $s\in [k-1]$ such that $j_s-j_{s-1} \equiv 0 \mod 2$.
    The first part of Lemma \ref{lem:cycles_distance} gives
    	\[
    	\dist_{P^{\Delta}}(u_{j_t}, u_{j_{t-1}}) \leq (k-1)^{i-1} \dist_{P^{\Delta,i-1}}(u_{j_t}, u_{j_{t-1}}) = (k-1)^{i-1}
    	\]
    for $s+1 \leq t \leq k-1$.
    In $P^\Delta$ we have $\dist_{P^{\Delta}}(u_{j_t}, u_{j_{t-1}}) = |j_t - j_{t-1}|$.
    Now the inductive hypothesis implies
    \[
    j_{k-1} = j_s + \sum_{t=s+1}^{k-1} j_t - j_{t-1} \leq (k-1)^{i-1} - 1 + (k-1-s) \cdot (k-1)^{i-1} \leq (k-1)^i - 1.
    \]\end{proof}

Recall the definition of $A_i$ in \eqref{eq:defofAi}.

    \begin{lem}\label{lem:cycles_analogueodd}
    Let $i \geq 1$ and $j,j'\in [\ell]$ with $j \not \equiv j' \mod 2$. 
    If $u_j\in\{v_j,w_j\}$, $u_{j'}\in\{v_{j'},w_{j'}\}$ and $u_ju_{j'}\in E(P^{\Delta,i})$, then $j-j' \in A_i$.
    \end{lem}

    \begin{proof}
    We induct on $i$ with $i\in\{ 1,2 \}$ being our base cases.

    $i=1$:
    Any path of length $k-1$ in $P^{\Delta, 0}$ whose endpoints form an odd edge  in $P^{\Delta,1}$  must not use $v_0w_0$ because of parity, and thus misses at least one of $v_0$,$w_0$.
    This implies $j-j' \in \{ -(k-1),-1,1,(k-1) \} \subseteq A_1$ (cf. base case of Lemma \ref{lem:gen_fromright}).

    $i=2$:
    If $u_ju_{j'}$ is present at time $1$ we are done because $A_1 \subseteq A_2$ and $j-j'\in A_1$ by the induction hypothesis.
    Suppose that the edge does not lie in $E(P^{\Delta,1})$.
    Let $Q = u_{j_0}\ldots u_{j_{k-1}}$ be a path in $P^{\Delta,1}$ with $j_0 = j'$, $j_{k-1} = j$ and $u_{j_t} \in \{ v_{j_t},w_{j_t} \}$ for $1 \leq t \leq k-2$.
    There has to be an even number of even edges in $Q$ because $j-j'$ is odd and $Q$ has odd length, and there cannot be more than two because the only even edges of $P^{\Delta,1}$ are $v_0w_0$, $v_0v_{k-2}$ and $w_0v_{k-2}$. 
    If all edges of $Q$ are odd we proceed as in the inductive step of Lemma \ref{lem:gen_fromright}.
    Otherwise there are precisely two even edges on $Q$.
    These two edges must share a common endpoint considering that the even edges in $P^{\Delta,1}$ form a triangle.
    Let $s\in [k-2]$ such that $u_{j_{s-1}}u_{j_s}$ and  $u_{j_s}u_{j_{s+1 }}$ are the even edges.
    We have either $j_{s-1} = j_{s+1} = 0$ or $\{ j_{s-1},j_{s+1} \} = \{ 0,k-2 \}$.
    For $t \in [k-1]\setminus\{ s,s+1 \}$, by induction choose $\alpha_t, \beta_t\in \NN_0$ such that $j_t - j_{t-1} = (k-1)- \alpha_t (k-2) - \beta_t k$. 
    This allows us to express $j_{k-1}-j_0$ as follows:
        \begin{linenomath}\begin{align*}
        j_{k-1} - j_0
        &= \sum_{t\in [k-1]\setminus\{ s,s+1 \}} j_t - j_{t-1} \; + \; j_s-j_{s-1} + j_{s+1} - j_s \\
        &= \sum_{t\in [k-1]\setminus\{ s,s+1 \}} \left((k-1)- \alpha_t (k-2) - \beta_t k\right) \; + \; j_{s+1} - j_{s-1} \\
        &= (k-3)\cdot (k-1)^1 - \sum_{t\in [k-1]\setminus\{ s,s+1 \}} \alpha_t (k-2) - \sum_{t\in [k-1]\setminus\{ s,s+1 \}} \beta_t k \; + \; j_{s+1} - j_{s-1} \\
        &= \begin{cases}
        (k-1)^2 - \sum_t \alpha_t (k-2) - \sum_t \beta_t k - 2(k-2) - k &,\text{ if } j_{s+1} - j_{s-1} = -(k-2); \\
        (k-1)^2 - \sum_t \alpha_t (k-2) - \sum_t \beta_t k - (k-2) - k &,\text{ if } j_{s+1} - j_{s-1} = 0; \\
        (k-1)^2 - \sum_t \alpha_t (k-2) - \sum_t \beta_t k - k &,\text{ if } j_{s+1} - j_{s-1} = k-2.
        \end{cases}
        \end{align*}\end{linenomath}
    Therefore $j_{k-1}-j_{0}=j-j'\in A_2$, as required.

    $i\geq 3$:
    We handle the case $u_ju_{j'}\in E(P^{\Delta,i-1})$ as before and so assume that $u_ju_{j'}$ is an odd edge not in $P^{\Delta,i-1}$.
    Let $u_{j_0}\ldots u_{j_{k-1}}$ be a $u_ju_{j'}$-path in $P^{\Delta,i-1}$ where $u_{j_t} \in \{ v_{j_t},w_{j_t} \}$ for $1 \leq t \leq k-2$, and let $J := \{ t \in [0,k-2]  : j_t \equiv j_{t+1} \mod 2 \}$.
    Since $j-j'$ and $k-1$ are odd, $|J|$ must be even.
    If $J$ is empty, that is, if $Q$ consists of odd edges we can again proceed as in Lemma \ref{lem:gen_fromright}.
    Suppose that $|J|\geq 2$ and let $s:=\min J$.
    Lemma \ref{lem:cycles_endpointofevenedge} yields $j_s \leq (k-1)^{i-1}-1$ while Lemma \ref{lem:cycles_distance} guarantees $j_t - j_{t+1} \leq (k-1)^{i-1}$ for $t\in [0,k-2]$.
    Therefore,
        \begin{linenomath}\begin{align*}
        j - j' = j_0 - j_{k-1} &\leq j_0 \\
        &= j_s + \sum_{t=0}^{s-1} j_t - j_{t+1} \\
        &\leq (k-1)^{i-1}-1 + s \cdot (k-1)^{i-1} \\
        &\leq (k-1)^{i-1}-1 + (k-3) \cdot (k-1)^{i-1} \\
        &= (k-1)^i - (k-1)^{i-1} - 1 \\
        &< (k-1)^i - F'(k-2,k).
        \end{align*}\end{linenomath}

    The last inequality uses \eqref{eq:FrobeniusDash} and $i\geq 3$.
    We now have $j-j'\in A_i$ by the definition of Frobenius numbers and \eqref{eq:FrobeniusDash} and because $(k-1)^i$ and $j-j'$ are odd.
    \end{proof}

   \begin{lem}\label{lem:lower_structure}
    Let $1\leq i < r$, and let $j,j'\in\{ 0,\ldots,\ell \}$ such that $j \equiv j' \mod 2$ and 
    $j+j' \geq (k-1)^i-(k-1)-2\cdot F'(k-2,k)-2$. If $u_j\in\{v_j,w_j\}$, $u_{j'}\in\{v_{j'},w_{j'}\}$ and $u_ju_{j'} \in E(P^{\Delta,i})\setminus E(P^{\Delta,i-1})$, then there exist $\alpha,\gamma\in\ZZ_{\geq -1}$, $\beta,\delta,\lambda,\mu\in\NN_0$ with $\lambda+\mu = (k-1)^{i-1}-1$ such that
        \[
        j = \lambda (k-1) - \alpha (k-2) - \beta k \qquad\mbox{ and }\qquad j' = \mu (k-1) - \gamma (k-2) - \delta k.
        \]
    \end{lem}
  
    \begin{proof}
    We induct on $i\geq 1$.
    
    Base case $i=1$:
    The only even edges in $E(P^{\Delta,1})\setminus E(P^{\Delta,0})$ are $v_0v_{k-2}$ and $w_0v_{k-2}$.
    Both of them satisfy the hypothesis $j+j' \geq (k-1)^1-(k-1)-2\cdot F'(k-2,k)-2$.
    The claim now holds with either $\alpha=-1$ and $\beta,\gamma,\delta,\lambda,\mu$ equal to zero or $\gamma=-1$ and $\alpha,\beta,\delta,\lambda,\mu$ equal to zero.
      
    Inductive step: Let $Q = u_{j_0}\ldots u_{j_{k-1}}$ be a path in $P^{\Delta, i-1}$ such that $j_0 = j$, $j_{k-1} = j'$, and $u_{j_t}\in\{ v_{j_t},w_{j_t} \}$ for $1\leq t \leq k-2$. We first show that $Q$ has exactly one even edge.
    The number of even edges in $Q$ is odd for otherwise we have $j \not \equiv j' \mod 2$.    
    If $i=2$ the only three even edges in $P^{\Delta,i-1}$ are $v_0v_{k-2}$, $v_{k-2}w_0$ and $v_0w_0$. A path cannot contain all three of them so $Q$ has precisely one even edge. 
    If $i\geq 3$, suppose there are at least three even edges in $Q$ and let $s,s'\in [k-1]$ such that $u_{j_s}u_{j_{s+1}}$ is the first and $u_{j_{s'-1}}u_{j_{s'}}$ is the last even edge in $Q$. Then $s+ (k-1-s') \leq k-4$. By Lemma \ref{lem:cycles_endpointofevenedge}
        \[
        j_s \leq (k-1)^{i-1}-1 \qquad,\qquad j_{s'} \leq (k-1)^{i-1}-1.
        \]
    Combining this with Lemma \ref{lem:cycles_analogueodd} and $\max A_{i-1} = (k-1)^{i-1}$ gives us
        \begin{linenomath}     \begin{align*}
        j+j'  = j_0 + j_{k-1} 
        &= \sum_{t=0}^{s-1} (j_t-j_{t+1}) + j_s + j_{s'} + \sum_{t=s'+1}^{k-1} (j_t - j_{t-1}) \\
            &\leq 2(k-1)^{i-1} - 2 + (s+ k-1-s')\cdot (k-1)^{i-1} \\
            &\leq (k-2)\cdot (k-1)^{i-1} - 2 \\
            &= (k-1)^i - (k-1)^{i-1} - 2 \\
            &< (k-1)^i-(k-1)-2\cdot F'(k-2,k)-2,
        \end{align*} \end{linenomath}
    which contradicts the assumption $j+j' \geq (k-1)^i-(k-1)-2\cdot F'(k-2,k)-2$.
    Here we used that $i\geq 3$ and so $(k-1)^{i-1} > 2\cdot F'(k-2,k) + (k-1)$ by \eqref{eq:FrobeniusDash}.
    We have thus shown that $Q$ has precisely one even edge.
  
    Take the unique $s^*\in [k-1]$ for which $u_{j_{s^*-1}}u_{j_{s^*}}$ is an even edge. We claim that there exist $\alpha^*,\gamma^*\in\ZZ_{\geq -1}$, $\beta^*,\delta^*,\lambda^*,\mu^*\in\NN_0$ such that $\lambda^*+\mu^* = (k-1)^{i-2}-1$ and
        \begin{equation} \label{jstar}
        j_{s^*-1} = \lambda^* (k-1) - \alpha^* (k-2) - \beta^* k \qquad,\qquad j_{s^*} = \mu^* (k-1) - \gamma^* (k-2) - \delta^* k.
        \end{equation}
    Indeed, this follows if $i=2$ and $j_{s^*-1}= j_{s^*}=0$ by setting all parameters to be $0$. For all other cases, this follows by the induction hypothesis. In order to appeal to the induction hypothesis, we need to establish the required lower bound on $j_{s^*-1} + j_{s^*}$ and show that the edge $u_{j_{s^*-1}}u_{j_{s^*}} \in E(P^{\Delta,i-1})\setminus E(P^{\Delta,i-2})$, which we now do. 
    We have
    	  \begin{linenomath}  \begin{align*}
    	j_{s^*-1} + j_{s^*}
    	&= j + j' - \sum_{t=1}^{s^*-1} (j_{t} - j_{t-1})  -  \sum_{t=s^*+1}^{k-1} (j_{t-1} - j_{t}) \\
    	&\geq (k-1)^i - (k-1) - 2F'(k-2,k) - 2 - (k-2)\cdot (k-1)^{i-1} \\
    	&= (k-1)^{i-1} - (k-1) -2 F'(k-2,k) - 2, 
    	\end{align*} \end{linenomath}
    as required.
    Now suppose for a contradiction that $u_{j_{s^*-1}}u_{j_{s^*}}$ already appeared at time $i-2$. If $i=2$, then the only even edge at time 0 is $v_0w_0$ and as we are not appealing to the induction hypthesis for this case, we can assume that $i\geq 3$.  Then by Lemma \ref{lem:cycles_endpointofevenedge}
    	\[
    	j_{s^*-1}, j_{s^*} \leq (k-1)^{i-2} -1 
    	\]
    and
    	\[
    	j_{s^*-1} + j_{s^*} \leq 2(k-1)^{i-2} - 2 < (k-1)^{i-1} - (k-1) - 2F'(k-2,k) - 2, 
    	\]
    contradicting our lower bound above, using that $i\geq 3$ and \eqref{eq:FrobeniusDash} here. 
    Therefore  $u_{j_{s^*-1}}u_{j_{s^*}} \in E(P^{\Delta,i-1})\setminus E(P^{\Delta,i-2})$ and the induction hypothesis gives \eqref{jstar}. 
    
    Now by Lemma \ref{lem:cycles_analogueodd} we have that we can find $\alpha_t,\beta_t \in \NN_0$ such that 
        \[
        j_t-j_{t-1} = (k-1)^{i-1} - \alpha_t (k-2) - \beta_t k
        \]
    for $s^* < t \leq k-1$ and
        \[
        j_{t}-j_{t+1} = (k-1)^{i-1} - \alpha_t (k-2) - \beta_t k
        \]
    for $0 \leq t < s^*-1$.
    Therefore,
        \begin{linenomath}     \begin{align*}
        j      &=  \sum_{t=0}^{s^*-2} (j_{t}-j_{t+1})  + j_{s^*-1} = \lambda (k-1) - \alpha (k-2) - \beta k, \\
        j'     &= \sum_{t=s^*+1}^{k-1}(j_{t}-j_{t-1}) +  j_{s^*} =\mu (k-1) - \gamma (k-2) - \delta k,
        \end{align*} \end{linenomath}
    where
        \begin{linenomath}     \begin{align*}
        \lambda &:= (s^*-1)\cdot (k-1)^{i-2} + \lambda^*,    &    &    \mu := (k-1-s^*)\cdot (k-1)^{i-2} + \mu^* ,&\\
        \alpha &:= \alpha_0 + \ldots + \alpha_{s^*-2}+\alpha^*, &    &     \beta := \beta_0 + \ldots + \beta_{s^*-2}+\beta^*, &\\
        \gamma &:= \alpha_{s^*+1}+\ldots+\alpha_{k-1}+\gamma^*,  &    &     \delta :=  \beta_{s^*+1}+\ldots+\beta_{k-1}+\delta^*. &
        \end{align*} \end{linenomath}
    Moreover,
        \[
        \lambda + \mu =  (k-2)(k-1)^{i-2} + \lambda^* + \mu^* = (k-1)^{i-1} - 1,
        \]
    which completes the induction.
    \end{proof}

We now complete the proof of the second part of Theorem \ref{thm:strategy_lowerbound}.
Take the smallest $i_0\in\NN$ for which the even edge $v_\ell w_\ell$ lies in  $E(P^{\Delta,i_0})$ and suppose that $i_0\leq r-1$.
Lemma \ref{lem:cycles_endpointofevenedge} and \eqref{eq:defofell} yield
    \[
    2(k-1)^{i_0}-2 \geq \ell + \ell = (k-1)^{r-1} - (k-1) - 2\cdot F'(k-2,k) - 2,
    \]
and so $i_0\geq r-1$ when $n$ and thus $r$ is sufficiently large. It remains to rule out the case $i_0=r-1$. Suppose that $i_0=r-1$. By Lemma \ref{lem:lower_structure} there exist $\alpha,\gamma\in\ZZ_{\geq -1}$, $\beta,\delta,\lambda,\mu\in\NN_0$ with $\lambda+\mu = (k-1)^{r-2}-1$ such that
    \begin{linenomath} \begin{equation}\label{eq:lower_structure}
    \ell = \lambda (k-1) - \alpha (k-2) - \beta k = \mu (k-1) - \gamma (k-2) - \delta k.
    \end{equation} \end{linenomath} 
By symmetry we can assume that $\lambda\leq\mu$. From \eqref{eq:lower_structure} and the definition of  $\ell$ in the statement of Theorem \ref{thm:strategy_lowerbound},  we obtain
    \begin{linenomath}     \begin{align}
    \label{eq:even_structure_alpha} F'(k-2,k) &= \left( \frac{(k-1)^{r-2}-1}2 - \lambda -1 \right)\cdot (k-1) + (\alpha+1)\cdot (k-2) + \beta k.   \end{align} \end{linenomath}
Using that $F'(k-2,k)$ is even \eqref{eq:FrobeniusDash}, if  we take \eqref{eq:even_structure_alpha} modulo 2 we can see that
    \begin{linenomath} \begin{equation}\label{eq:lower_modulo}
    \Lambda:= \frac{(k-1)^{r-2}-1}2 - \lambda \equiv 1 \mod 2.
    \end{equation} \end{linenomath} 
The condition $\lambda+\mu = (k-1)^{r-2}-1$ and the assumption $\lambda \leq \mu$ imply
    \begin{linenomath} \begin{equation}\label{eq:lower_pigeon}
    \lambda \leq \frac{(k-1)^{r-2}-1}2 \leq \mu.
    \end{equation} \end{linenomath} 
We cannot have equality in \eqref{eq:lower_pigeon} because of \eqref{eq:lower_modulo}. Therefore $\Lambda\geq 1$.
Since $2(k-1)$ can be written as $(k-2) + k$, by \eqref{eq:even_structure_alpha} we have
    \[
    F'(k-2,k) = \left( \alpha+1 + \frac 1 2 \left(\Lambda -1 \right) \right)\cdot  (k-2) +  \left( \beta + \frac 1 2 \left(\Lambda - 1\right) \right) \cdot k.
    \]
However, this contradicts the definition of $F'(k-2,k)$ \eqref{eq:FrobeniusDash}. Consequently, $v_\ell w_\ell \notin E(P^{\Delta,r-1})$.

\section{Upper bounds}\label{sec:cycles_upper}
We start with some general results on $C_k$-processes on paths in Section \ref{sec:cycles_upper_paths}.
We will prove parts \ref{upperpart1} and \ref{upperpart2} of Theorem \ref{thm:strategy_upperbound} in Section \ref{sec:cycles_upper_part_i_and_ii}.
Part \ref{upperpart3} of Theorem \ref{thm:strategy_upperbound} will be shown in Section \ref{sec:cycles_upper_part_iii}

\subsection{Results on paths}\label{sec:cycles_upper_paths}
Let $n'\in \NN$, and $(P_{n'}^i)_{i\geq 0}$ be the $C_k$-process on $P_{n'}$.
We write $P^i$ instead of $P_{n'}^i$ when $n'$ is clear from context.
The sets
    \begin{linenomath}     \begin{equation*}
    D_i = D_i(n') := \{  \ell \in [n'-1] : xy \in E(P^i) \text{ whenever } y-x = \ell \} 
    \end{equation*} \end{linenomath}
play a central role in proving upper bounds on $\tau_{C_k}(P_{n'})$.
Clearly $D_i \subseteq D_{i+1}$.
If $D_i =[n'-1]$, then the percolation process is over by the $i^{\text{th}}$ step. When $k$ is even then the process is already over when $D_i$ contains just the odd integers up to $n'-1$, since then $P^i$ has stabilised at $K_{\floor{n'/2},\ceil{n'/2}}$.
Flipping the vertices of $P_{n'}$, i.e. the map $\sigma: V(P_{n'})\to V(P_{n'})$, $x \mapsto n'-1-x$, is an automorphism of $P_{n'}$ and hence of $P^i$ for all $i\geq 0$ by Observation \ref{obs:hom}.
For any $x,y \in V(P_{n'})$ one has $x+y \leq n'-1$ or $\sigma(x)+\sigma(y) \leq n'-1$.
This allows us to write $D_i$ as follows:
    \begin{equation}\label{eq:defofDi}
    D_i = \left\lbrace  \ell \in [n'-1] : xy \in E(P^i) \text{ whenever } y-x = \ell \text{ and } x+y \leq n'-1 \right\rbrace
    \end{equation}
The next  lemmas state further simple properties about how $D_i$ develops during the $C_k$-process. Recall the definition of a sumset $hA$ from Section \ref{sec:notation}.

    \begin{lem}\label{lem:gen_iteration}
    For every $i\geq 0$, $(k-1)D_i \cap [n'-1]\subseteq D_{i+1}$. 
    \end{lem} 
        
    \begin{proof}
    Let $\ell\in(k-1)D_i \cap [n'-1]$. Choose $d_1,\ldots,d_{k-1}\in D_i$ such that $\ell = d_1+\ldots+d_{k-1}$. For any $x\in V(P_{n'})$ with $x+\ell\in V(P_{n'})$,
    \[ x,x+d_1,\ldots,x+ (d_1+\ldots+d_{k-1})\]
    is a path of length $k-1$ from $x$ to $x+\ell$ in $P^i$ since $d_1,\ldots,d_{k-1}\in D_i$. Thus $x$ and $x+\ell$ are adjacent in $P^{i+1}$ and the claim follows by the definition of $D_{i+1}$.
    \end{proof}
    
    \begin{lem}\label{lem:gen_firstkodd}
    If $n' \geq 3(k-1)$, then $\{ \ell \in [k] : \ell \text{ odd } \}\subseteq D_2$.
    \end{lem}
    
    \begin{proof}
    We have $D_0 = \{ 1 \}$ and $D_1 = \{ 1,k-1 \}$. For every odd $3\leq\ell\leq k$ and $x\in V(P_{n'})$ with $x+(x+\ell) \leq n'-1$,
        \[
        x,\ldots,x+\frac{\ell-1}2, x+(k-1)+\frac{\ell-1}2, x+(k-1)+\frac{\ell-1}2-1,\ldots,x+\ell
        \]
    is a path of length $k-1$ from $x$ to $x+\ell$ in $P^1$ due to the hypothesis $3(k-1) \leq n'$. Indeed, a quick analysis of two cases $x\leq (k-1)$ and $x>(k-1)$ gives that $x+(k-1) +(\ell-1)/2\leq n'-1$.  
    Therefore \eqref{eq:defofDi} gives $\{ \ell\in [k] : \ell \text{ odd } \} \subseteq D_2$. 
    \end{proof}

Recall that Lemma \ref{lem:gen_fromright} in the proof of Theorem \ref{thm:strategy_lowerbound} indicated how differences that occur as edges at time $i$ look like and thereby helped us to obtain lower bounds on the running time.
For upper bounds we need a converse statement telling us for which parameters $\alpha,\beta$ the differences $(k-1)^i - \alpha\cdot (k-2) - \beta\cdot k$ {\em do} appear.
To this end we define a subset
    \begin{linenomath}     \begin{align*}
    A'_i &:= \left\lbrace (k-1)^i-\alpha\cdot(k-2)-\beta\cdot k : \alpha,\beta\in\NN_0, \alpha+\beta\leq (k-1)^{i-2}\cdot(k-2) \right\rbrace 
    \end{align*} \end{linenomath}
of $A_i =  \left\lbrace (k-1)^i-\alpha\cdot(k-2)-\beta\cdot k : \alpha,\beta\in\NN_0 \right\rbrace$.
This set lies in the intersection of $A_i$ and the interval $[(k-1)^{i-2},(k-1)^i]$. The upper end of the interval is attained when $\alpha=\beta=0$ whereas the lower end is achieved by $\alpha=0$, $\beta=(k-1)^{i-2}\cdot(k-2)$.
The next lemma states that a slightly smaller interval piece of $A_i$  is fully contained in $A_i'$. 

    \begin{lem}\label{lem:gen_intervalfromright}
    For every $i\geq 3$,
        \[
        [(k-1)^{i-2}+2(k-1), (k-1)^i] \cap A_i  \subseteq A'_i.
        \]
    \end{lem}

    \begin{proof}
    Let $i\geq 3$ and $\ell \in  [(k-1)^{i-2}+2(k-1), (k-1)^i] \cap A_i$. Then there exist $\alpha,\beta\in\NN_0$ satisfying $\ell = (k-1)^{i}-\alpha(k-2)-\beta k$. We may assume that $\alpha\leq k-1$ because $(\alpha-k)\cdot(k-2) + (\beta+k-2)k = \alpha (k-2) + \beta k$.
    From
        \[
        (\alpha+\beta)k - 2\alpha = (k-1)^{i}-\ell \leq  (k-1)^i-(k-1)^{i-2}-2(k-1) = (k-1)^{i-2}\cdot (k-2)k -2(k-1)
        \]
    we infer
        \[
        \alpha+\beta \leq (k-1)^{i-2}\cdot (k-2) + \frac{2\alpha}k - \frac{2(k-1)}k \leq (k-1)^{i-2}\cdot (k-2), 
        \] 
    hence $\ell \in A'_{i}$.
    \end{proof}

The next lemma ensures that the relevant piece of $A_i'$ is contained in $D_i$. This fact will play a crucial role in us showing that a bootstrap process has ended. 
In the proof the advantage of the somewhat technical choice of the upper bound on $\alpha +\beta$ in the inductive proof becomes visible. 

    \begin{lem} \label{lem:gen_revers}
    Given $n'\geq 3(k-1)$ we have that
        \[ 
        A'_i \cap [n'-1]\subseteq D_i.
        \]    
    for every $i\geq 0$.
    \end{lem}
        
    \begin{proof}
    We induct on $i\geq 0
    $. We have that $A'_0 = \{ 1 \} = D_0$ and $A'_1 = \{ k-1 \} \subseteq D_1$.
    Let $i=2$, and let $x,y\in V(P_{n'})$ with $y-x = (k-1)^2-\alpha(k-2)-\beta k$ for some $\alpha+\beta \leq (k-2)$. Write $s:=k-1-\alpha-\beta$ and note that $s\geq 1$.
    If $x \geq \beta$ then
        \[
        x, x-1,\ldots, x-\beta, x-\beta+(k-1),\ldots, x-\beta+s(k-1),x-\beta+s(k-1)+1,\ldots,y
        \] 
    is a path of length $k-1$ from $x$ to $y = x-\beta+s(k-1)+\alpha$ in $P^1$.
    If $x \leq \beta$ consider the $xy$-path
        \[
        x,\ldots, x+\alpha, x+\alpha+(k-1),x+\alpha+(k-1)-1,\ldots, x+\alpha+(k-1)-\beta, x+\alpha+2(k-1)-\beta,\ldots, y.\]
    This path is well-defined because all vertices lie in $[n'-1]$.
    Indeed, for $s\geq 2$ the largest term in the sequence is $y$, which lies in $V(P_{n'})$ by assumption, whereas if $s=1$, then $x+\alpha+(k-1)-\beta = y$ and the maximum is $x+\alpha+(k-1)$ which can be bounded from above by $\beta +\alpha + (k-1) \leq 2k-3 \leq n'-1$. 
    Thus, $xy \in E(P^2)$.
    Since $x$ and 
    $y$ were arbitrary, $A'_2 \cap [n'-1] \subseteq D_2$.
    For every $i\geq 3$, the induction hypothesis and Lemma \ref{lem:gen_iteration} imply 
        \[    
        A'_i \cap [n'-1] \subseteq (k-1)A'_{i-1}\cap [n'-1] \subseteq (k-1)D_{i-1} \cap [n'-1] \subseteq D_i ,
        \]
    where the inclusion $A'_i \subseteq (k-1)A'_{i-1}$ follows from the fact that for any $\alpha,\beta\in\NN_0$ with $\alpha+\beta\leq (k-1)^{i-2}\cdot(k-2)$ we can find
        \[
        \alpha_1,\ldots,\alpha_{k-1} \in \left\lbrace \floor*{\frac{\alpha}{k-1}}, \ceil*{\frac{\alpha}{k-1}}\right\rbrace \quad,\quad
        \beta_1,\ldots,\beta_{k-1} \in \left\lbrace \floor*{\frac{\beta}{k-1}}, \ceil*{\frac{\beta}{k-1}}\right\rbrace
        \]
    such that $\alpha = \alpha_1+\ldots+\alpha_{k-1}$, $\beta = \beta_1+\ldots+\beta_{k-1}$ and $\alpha_s+\beta_s \leq (k-1)^{i-3}\cdot (k-2)$, $1\leq s\leq k-1$.
    \end{proof}
    
    \begin{prop}\label{prop:odd_path}
    If $k\geq 3$ is odd and $3(k-1) \leq n' \leq (k-1)^\rho - F(k-2,k)$ for some integer $\rho\geq 4$ then $P_{n'}^\rho$ is the complete graph on $n'$ vertices.
    \end{prop}
        
    \begin{proof}
    Our goal is to show $D_\rho = [n'-1]$. To do so we write $[n'-1]$ as the union
        \[
        [n'-1] = [3(k-1)-1] \;\cup\; [3(k-1),n'-1]
        \]
    and show $[3(k-1)-1] \subseteq D_4$ and $[3(k-1),n'-1]\subseteq D_\rho$.
    Lemma \ref{lem:gen_firstkodd} yields $k-2\in D_2$ because $k-2$ is odd. Then for each even $2\leq\ell \leq k-1$ and each vertex $x$ with $x+(x+\ell) \leq n'-1$,
    \[ 
    x,\ldots,x+\frac{\ell}2, x+(k-2)+\frac{\ell}2, x+(k-2)+\frac{\ell}2-1,\ldots,x+\ell 
    \]
    is a path of length $k-1$ from $x$ to $x+\ell$ in $P_{n'}^2$.
    Here we used  that $x+(k-2)+\ell/2\leq n'-1$, which follows from a quick case analysis of $x\leq k-2$ and $x\geq k-1$,  so all vertices of the path indeed belong to $P_{n'}$.
    Now \eqref{eq:defofDi} implies $[k] \subseteq D_3$. Applying Lemma \ref{lem:gen_iteration} gives 
        \[
        D_4 \supseteq [(k-1)\cdot k]\cap [n'-1] \supseteq [3(k-1)-1].
        \]
    The inclusion $[3(k-1), n'-1] \subseteq D_\rho$ follows from Lemmas \ref{lem:gen_intervalfromright} and \ref{lem:gen_revers}. Indeed, by the definition of the Frobenius number $F(k-2,k)$ \eqref{eq:Frobenius},  
     and as $k$ is odd we have $A_i \supseteq (-\infty, (k-1)^i-F(k-2,k)-1]$ for all $i\geq 0$. This allows us to write
        \begin{linenomath}     \begin{align*}
            [3(k-1),n'-1] &= [n'-1] \cap [3(k-1),(k-1)^\rho-F(k-2,k)-1] \\
            &= [n'-1] \cap \bigcup_{i=3}^\rho [(k-1)^{i-2}+2(k-1), (k-1)^i-F(k-2,k)-1] \\
            &\subseteq [n'-1] \cap \bigcup_{i=3}^\rho [(k-1)^{i-2}+2(k-1), (k-1)^i] \cap A_i \\
            &\subseteq [n'-1] \cap \bigcup_{i=3}^\rho A'_i 
            \subseteq \bigcup_{i=3}^\rho D_i = D_\rho.
        \end{align*} \end{linenomath}
    The second equality holds since $(k-1)^{3-2}+2(k-1) = 3(k-1)$ and $(k-1)^i-F(k-2,k) \geq (k-1)^{i+1-2}+2(k-1)$ for $i\geq 3$ by \eqref{eq:Frobenius}. We also used Lemma \ref{lem:gen_intervalfromright} followed by Lemma \ref{lem:gen_revers} in the fourth line.
    We have shown that $D_\rho = [n'-1]$, so $P_{n'}^\rho$ is a complete graph. 
    \end{proof}

The bipartite version of Proposition \ref{prop:odd_path} reads as follows.
    
    \begin{prop}\label{prop:even_path}
    If $k\geq 4$ is even and $3(k-1) \leq n' \leq (k-1)^\rho$ for some $\rho\in\NN$ then any $x,y \in V(P_{n'})$ with $|x-y|\in A_\rho$ are adjacent in $P_{n'}^\rho$.
    This implies that if $n' \leq (k-1)^\rho - F'(k-2,k)$, $P_{n'}^\rho$ is copy of $K_{\floor{n'/2},\ceil{n'/2}}$.
    \end{prop}

    \begin{proof}
    Since $n' \geq 3(k-1)$ we may invoke Lemma \ref{lem:gen_firstkodd} to obtain
        \[
        \{ \ell \in [k] : \ell \text{ odd} \} \subseteq D_2.
        \]
    Lemma \ref{lem:gen_iteration} then gives 
        \[
        \{ \ell \in [3(k-1)-1] : \ell \text{ odd} \} \;\subseteq\; \{ \ell \in [(k-1)^2] : \ell \text{ odd} \} \cap [n'-1] \subseteq D_3 \subseteq D_\rho.
        \]
    This takes care of the case $|x-y| < 3(k-1)$.
    Suppose that $|x-y| \geq 3(k-1)$.

    The intervals $[(k-1)^{i-2}+2(k-1), (k-1)^i-F'(k-2,k)]$ and $[(k-1)^{i-1}+2(k-1), (k-1)^{i+1}-F'(k-2,k)]$ intersect whenever $i\geq 3$.
    Therefore by \eqref{eq:FrobeniusDash}
        \[
        [3(k-1),(k-1)^\rho] = [(k-1)^\rho - F'(k-2,k), (k-1)^\rho] \cup \bigcup_{i=3}^{\rho} [(k-1)^{i-2}+2(k-1), (k-1)^i-F'(k-2,k)-1].
        \]
    If $(k-1)^\rho - F'(k-2,k) < |x-y|$ we can use Lemma \ref{lem:gen_intervalfromright} because $|x-y|\in A_\rho$ and thereby obtain $|x-y|\in A'_\rho$.  Lemma \ref{lem:gen_revers} then tells us that $|x-y| \in D_\rho$.
    If not, there exists $3\leq i\leq \rho$ such that $|x-y| \in [(k-1)^{i-2}+2(k-1), (k-1)^i-F'(k-2,k)-1]$ so we have $|x-y| \in A_i$ by definition of $F'(k-2,k)$ and we can apply Lemmas \ref{lem:gen_intervalfromright} and \ref{lem:gen_revers} to conclude $|x-y|\in D_i \subseteq D_\rho$.
    In the case that $n' \leq (k-1)^\rho - F'(k-2,k)$ a copy of $K_{\floor{n'/2},\ceil{n'/2}}$ is present at time $\rho$. Since $P_{n'}$ is bipartite, so is $\final{P_{n'}}_{C_k}$ due to Lemma \ref{lem:cycles_bipartiteness}. Thus the process has stabilised.
    \end{proof}

This completes our preliminary investigation of the $C_k$-process on paths.

\subsection{Proof of parts \ref{upperpart1} and \ref{upperpart2} of Theorem \ref{thm:strategy_upperbound}}\label{sec:cycles_upper_part_i_and_ii}
Suppose that $k$ is odd, and let $x,y\in V(G)$.
If there exists an $xy$-path $Q$ of length at least $3(k-1)-1$ it satisfies the hypotheses of Proposition \ref{prop:odd_path} with $s=r$ and $n'$ being the length of $Q$.
In that case the vertices of $Q$ must form a clique in $G_r$.
In particular $x$ and $y$ are adjacent.
If every path from $x$ to $y$ in $G$ has length less than $3(k-1)-1$, invoke Lemma \ref{lem:cycles_cycleattimetwo} to fix a cycle $C\subseteq G_2$ containing $x$, a shortest path $Q$ from $y$ to $V(C)$ in $G$ and an arbitrary vertex $z\in V(G)\setminus V(C)$ with $N_G(z) \cap V(C) \neq \varnothing$.
The latter vertex exists because $G$ was assumed to be connected, and is needed because we cannot rule out that $y \in Q$ and Lemma \ref{lem:cycles_smalldistance} requires a $(k+1)$-vertex graph.
Apply Lemma \ref{lem:cycles_smalldistance} to $G_2[V(C)\cup V(Q) \cup \{z\}]$.
As $\tau_{C_k}(G_2[V(C)\cup V(Q) \cup \{z\}])$ is trivially bounded by $(|V(C)|+|V(Q)|+1)^2/2$ we obtain that $xy\in E(G_r)$ when $n$ and hence $r$ is sufficiently large.

We will see that part \ref{upperpart2} is analogous to part \ref{upperpart1} with the roles of cliques being played by complete bipartite graphs. So suppose now  that $k$ is even and $G$ is bipartite with partite sets $X$, $Y$, and let $x\in X$, $y\in Y$.
If there is a path of length at least $3(k-1)-1$ from $x$ to $y$, then Proposition \ref{prop:even_path} with $\rho=r$ and $n'$ the length of that path implies $xy \in E(G_r)$.
Should there be no such path, Lemma \ref{lem:cycles_cycleattimetwo} again allows us to choose a cycle $C\subseteq G_2$ containing $x$, a shortest path $Q$ from $y$ to $V(C)$ in $G$, and $z\in V(G)\setminus V(C)$ such that $N_G(z) \cap V(C) \neq \varnothing$.
Lemma \ref{lem:cycles_smalldistance} applied to $G_2[V(C)\cup V(Q)\cup \{z\}]$ tells us that at time $r$ each vertex of $X\cap (V(C\cup Q)\cup \{z\})$ neighbours each vertex of $Y\cap (V(C\cup Q)\cup \{z\})$.

\subsection{Proof of part \ref{upperpart3} of Theorem \ref{thm:strategy_upperbound}}\label{sec:cycles_upper_part_iii}
Assume that in the following $k$ is even and $G$ is not bipartite.    
When we dealt with the upper bound for odd cycles and wanted to show that an edge $xy$ from the final graph occurs at a certain time in the process it was sufficient to restrict ourselves to an $xy$-path in the starting graph. 
In the case of even $k$ and non-bipartite $G$ one has to modify the approach since all $xy$-paths in $G$ could have even length while the final graph of the $C_k$-process on a path is not a clique but a complete bipartite graph (cf. Proposition \ref{prop:even_path}), and thus the restricted $C_k$-process does not yield the desired edge.
To deal with this issue we consider carefully chosen odd walks instead of odd paths. These odd walks  will be chosen so that they contain sufficiently long subwalks without repeated vertices.
More precisely we restrict our attention to odd walks which can be expressed as the union of two paths as specified in the next claim:

    \begin{clm}\label{clm:even_shortestwalk}
    Let $e\in E(\final{G}_{C_k})$.
    Then there exist $\ell,\ell'\in\NN_0$ with $\ell'\leq\ell\leq n-2$, $\ell'\leq n-3$, and vertices $v_0,\ldots,v_\ell,w_0,\ldots,w_{\ell'}\in V(G)$ such that $w_0v_0\ldots v_\ell$ and $v_0w_0\ldots w_{\ell'}$ are paths, $v_\ell\ldots v_0w_0\ldots w_{\ell'}$ is a shortest odd walk between the endpoints of $e$ in $G$, and $v_{j}\neq w_{j'}$ for $j\neq j'$.
    \end{clm}
    
    \begin{proof}
    Take a shortest odd walk $u_0\ldots u_m$ between the endpoints of $e$ in $G$. The claim is clearly satisfied with $\ell'=0$ and $\ell=m-1$ when the shortest odd walk is already a path. 
    We therefore assume that the walk has at least one repeated vertex. Observe that for any $0 \leq j < j'\leq m$ with $u_j = u_{j'}$, $j-j'$ must be odd for otherwise $u_0\ldots u_j u_{j'+1}\ldots u_m$ would be a shorter odd walk. 
    This implies that no vertex occurs more than twice on $u_0\ldots u_m$.    
    Set
        \begin{linenomath}     \begin{align*}
        j_0 &:= \max\{ j \in [m] \mid \exists j'>j : u_j = u_{j'} \} \qquad \mbox{ and } \qquad
        j_1 &:= \min\{ j \in [m] \mid \exists j'<j : u_j = u_{j'} \},
        \end{align*} \end{linenomath}
    and let $j'_0,j'_1\in [m]$ be the unique integers satisfying $j'_0 > j_0, u_{j_0} = u_{j'_0}$ and $j'_1 < j_1, u_{j_1} = u_{j'_1}$.
    Then $j_0 < j_1$ since otherwise $u_0\ldots u_{j'_1-1}u_{j_1}\ldots u_{j_0}u_{j'_0+1}\ldots u_m$ would be an odd walk of length less than $m$.
    By the extremality of $j_0$ and $j_1$ we have $j'_1 \leq j_0$ and $j_1 \leq j'_0$.
    In fact, we have that equality is attained in both of the last two inequalities. 
    Indeed, if one of them was strict the walk $u_0\ldots u_{j'_1-1}u_{j_1}\ldots u_{j_0}u_{j'_0+1}\ldots u_m$ would have length
        \[
        j'_1 + (j_1-j_0) + m-j'_0 < j'_1 + (j'_0-j'_1) + m - j'_0 = m
        \]
    so
        \[
        0 \equiv j'_1 + (j_1-j_0) + m-j'_0 \equiv (j_1-j'_1) + (j'_0-j_0) +m \mod 2
        \]
    by the minimality of $m$. But this would contradict the fact that $j_1-j'_1$, $j'_0-j_0$ and $m$ are all odd.
    Therefore $j'_1 = j_0$ and $j'_0 = j_1$, which implies $u_{j_0} = u_{j_1}$. We conclude $j_1-j_0 \geq 3$  and $j_1+j_0=j_1-j_0+2j_0$ is odd because $j_1-j_0$ must be odd. 
    By definition of $j_0$ and $j_1$, both $u_0\ldots u_{j_1-1}$ and $u_{j_0+1}\ldots u_m$ are paths and each of the vertices $u_{j_0+1},\ldots,u_{j_1-1}$ occurs precisely once on $u_0\ldots u_m$.
    Define $v_j$, $0\leq j \leq \ell:=\floor{(j_0+j_1)/2}$, and $w_{j'}$, $0\leq j' \leq \ell' := m - \ceil{(j_0+j_1)/2}$ by
        \[
        v_j := u_{\floor*{\frac{j_0+j_1}2}-j} \quad,\quad w_{j'} := u_{\ceil*{\frac{j_0+j_1}2}+j'}.
        \]
    Then both $v_\ell \ldots v_0w_0$ and $v_0w_0\ldots w_{\ell'}$ are paths. Therefore $\ell,\ell' \leq n-2$. Now we check that $\ell' \leq \ell$ and $v_j\neq w_{j'}$ whenever $j\neq j'$.
    Suppose there were $j\neq j'$ with $v_j = w_{j'}$.
    Due to the definition of $v_j, w_{j'}$ and the minimality of $m$ we have
       \[
       \left(\floor*{\frac{j_0+j_1}2}-j\right) - \left(\ceil*{\frac{j_0+j_1}2}+j'\right) \equiv 1 \mod 2,
      \]
    and thus, as $j_0 - j_1 \equiv 1 \!\mod 2$, 
    	\[
    	j_0 - \floor*{\frac{j_0+j_1}2}+j \equiv \ceil*{\frac{j_0+j_1}2}+j'-j_1 \mod 2.
    	\]
    But then replacing the longer of the two walks $u_{j_0}\ldots u_{\floor*{\frac{j_0+j_1}2}-j}$ and $u_{j_1}\ldots u_{\ceil*{\frac{j_0+j_1}2}+j'}$ by the shorter one creates an odd walk between the endpoints of $e$ whose length is less than $m$. Here we used these walks are not the same size due to the fact that $j\neq j'$. 
    If $\ell' \leq \ell$ we are done.
    Otherwise we simply relabel the path by interchanging the roles of $\ell$ and $\ell'$ and turning $v_i$ into $w_i$ and vice versa.  Finally,  we cannot have $\ell' = \ell = n-2$ as in that case $w_{\ell'}\notin \{ v_0,\ldots,v_\ell \}$ gives $|\{ v_0,\ldots,v_\ell,w_0,w_{\ell'} \}| = n+1$.
    \end{proof}

    \begin{rem} \label{rem:no intersect}
    The property $v_j \neq w_{j'}$ for $j\neq j'$ in Claim \ref{clm:even_shortestwalk} guarantees that $\{ v_j : j\in J\}\cap \{ w_{j'} : j'\in J' \} = \emptyset$ whenever $J\subset[\ell]$ and $J'\subset[\ell']$ are disjoint.
    We do not require additional assumptions on the indices $j$ for which $v_j=w_j$.
    However, note that we can assume that there is  $\ell_0, \ell_1\in[\ell']$ such that $v_j = w_j$ when $\ell_0 \leq j \leq \ell_1$ and $v_j\neq w_j$ otherwise.
    Indeed,  we  can take $\ell_0 := \min\{j: v_j = w_j\}$ and $\ell_1 := \max\{j: v_j = w_j\}$ and replace $w_{\ell_0}\ldots w_{\ell_1}$ by $v_{\ell_0}\ldots v_{\ell_1}$.
    Two examples of such odd walks are depicted in Figure \ref{fig:oddwalks}.
    \end{rem}

    \begin{figure}[h]
    \centering
    \includegraphics[width=\linewidth]{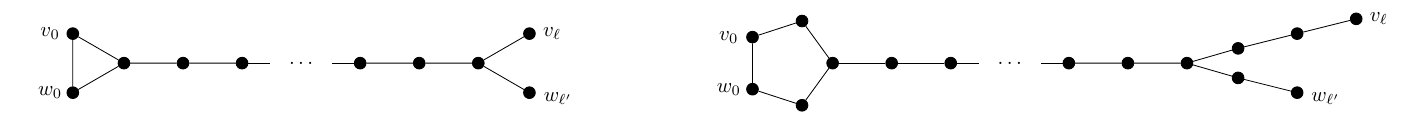}
    \caption{Two potential shortest odd walks between vertices $v_\ell$ and $w_{\ell'}$.}
    \label{fig:oddwalks}
    \end{figure}

Let $x,y\in V(G)$ be distinct vertices.
If the length of a shortest odd walk between them is smaller than $4(k-1)^2$ we can fix a non-bipartite subgraph of $G_2$ on at least $k+1$ vertices that contains $x$, $y$ and a $k$-cycle (the existence of such a subgraph is guaranteed by Lemma \ref{lem:cycles_cycleattimetwo}) and apply Lemma \ref{lem:cycles_smalldistance}.
From now on let the length of a shortest odd walk from $x$ to $y$ in $G$ be at least $4(k-1)^2$.
We are done once we have shown $xy\in E(G_r)$.
Let $v_\ell\ldots v_0w_0\ldots w_{\ell'}$ be a shortest odd walk from $x$ to $y$ or $y$ to $x$ as given by Claim \ref{clm:even_shortestwalk}. Note that $\ell \equiv \ell' \mod 2$ and $v_\ell \neq w_{\ell'}$ because the endpoints of the path are assumed to be distinct. The remaining proof is divided into the Claims \ref{clm:even_midwalk}, \ref{clm:even_midedge} and \ref{clm:even_longwalk}.
    
    \begin{clm}\label{clm:even_midwalk}
    If $\ell+\ell' \leq (k-1)^{r}-(k-1)\cdot F'(k-2,k) - 3(k-1)^2$, we have $xy\in E(G_{r})$.
    \end{clm}
    
    \begin{proof}
    Write $\ell'=q'(k-1)+s'$ and $\ell = q(k-1)+s$ and where $q,q',s,s'\in\NN_0$, $0 \leq s, s' \leq k-2$.  Recall that $\ell \geq \ell'$ and $\ell+\ell' \geq 4(k-1)^2-1$, hence $q \geq q'$ and $q > 1$.
    At time $1$,
        \[
        P := w_0\ldots w_{s'}w_{s'+(k-1)}\ldots w_{s'+q'(k-1)}
        \]
    is a path of length $q'+s'$ from $w_0$ to $w_{\ell'}$. It contains neither $v_0$ nor $v_\ell$ as $\ell'\leq\ell$, $w_{\ell'}\neq v_{\ell}$, and $w_j\notin \{v_0, v_\ell\}$ for $j<\ell$ by Claim \ref{clm:even_shortestwalk}.
    If $s'=0$,
        \[
        Q_0 := v_0v_{1}v_{1+(k-1)}\ldots v_{1+q(k-1)}v_{1+q(k-1)+1}\ldots v_{1+q(k-1)+s-1}
        \]
    is a path of length $q+s$ from $v_0$ to $v_\ell$ which is vertex-disjoint from $P$ as the indices of the vertices on $P$ are multiples of $k-1$ whereas the indices of vertices on $Q_0-v_0-v_\ell$ are not.
    The union of $P$, $Q_0$ and $v_0w_0$ is an $xy$-path  in $G_1$ of length $q'+(q+s)+1$. Note that
        \[
        q'+(q+s)+1 \equiv q'(k-1)+ q(k-1)+s+1 \equiv (\ell - \ell'+1) \equiv 1 \mod 2
        \]
    and
        \[
        q' + q+s + 1 \leq \frac{\ell'+\ell}{k-1} + k - 1 \leq (k-1)^{r-1} - F'(k-2,k) - 2(k-1) < (k-1)^{r-1} - F'(k-2,k).
        \]
   
    If $s' >0$, recall that $q>1$ and consider the path
        \[
        Q_{s'} := \begin{cases}
        v_0v_{k-1}\ldots v_{q(k-1)}v_{q(k-1)-1}\ldots v_{(q-1)(k-1)+s}v_\ell &, \; \text{if } s > s' ; \\
        v_0v_{k-1}\ldots v_{q(k-1)}v_{q(k-1)+1}\ldots v_\ell &, \; \text{if } s\leq s'.
        \end{cases}
        \]
    This path has length $ q+(k-1)-s+1$ or $q+s$.
    It is vertex-disjoint from $P$ since $w_j \equiv s'\mod k-1$ for all $j\geq s'$ with $w_j \in V(P)$, whereas all $j\in [0,\ell']$ with $v_j \in V(Q_{s'})\setminus\{v_0,v_\ell\}$ satisfy $j > s'$ and $j\not \equiv s'\mod k-1$.
    Moreover, recall that $v_0$ and $v_\ell$ do not lie on $P$.
        
    As $v_0w_0 \in E(G_1)$, the union of the paths $P$,$Q_{s'}$ and the edge $v_0w_0$ is an $xy$-path of length $q'+s'+q+k-s+1$ or $q'+s' + q + s+1$, where
        \[
         q'+s' + q+k- s + 1\;\equiv\; q'+s'+q+s  + 1\;\equiv\; \ell+\ell'+1 \mod 2.
        \]
    The length of $P\cup Q_{s'} \cup \{v_0w_0\}$ is bounded from above by
        \[
        \frac{\ell+\ell'}{k-1} + 2(k-2) + 1 < (k-1)^{r-1} - F'(k-2,k).
        \]
    In all cases we have an odd $xy$-path of length at least $q+q'$ and less than $(k-1)^{r-1}- F'(k-2,k)$. Using
        \[
        q + q' = \floor*{\frac{\ell}{k-1}} + \floor*{\frac{\ell'}{k-1}} \geq \floor*{\frac{\ell+\ell'}{k-1}} - 1 \geq 3(k-1)
        \]
    we can apply Proposition \ref{prop:even_path} with $\rho=r-1$ and Observation \ref{obs:hom} to either $P \cup Q_0 \cup \{ v_0w_0\}$ or $P\cup Q_{s'}\cup \{v_0w_0\}$ to deduce $xy\in E(G_{1+r-1}) = E(G_r)$.
    \end{proof}

    \begin{clm}\label{clm:even_midedge}
    Set $w_{-1} := v_0$. Then
        \[
        v_{\frac{(k-1)^i-(k-1)}{2}}w_{\frac{(k-1)^i-(k-1)}{2}+k-2} \in E(G_i),
        \]
    whenever $i\geq1$ with $\frac{(k-1)^i-(k-1)}{2} \leq \ell$ and $\frac{(k-1)^i-(k-1)}{2}+k-2 \leq \ell'$, and
        \[
        v_{\frac{(k-1)^i-(k-1)}{2}+k-1}w_{\frac{(k-1)^i-(k-1)}{2}-1} \in E(G_i)
        \]
    whenever $i\geq 1$ with  $\frac{(k-1)^i-(k-1)}{2}+(k-1) \leq \ell$ and $\frac{(k-1)^i-(k-1)}{2}-1 \leq \ell'$.
    \end{clm}
        
    \begin{proof}
    The size constraints are only necessary to guarantee that the vertices occurring in the statement actually exist.
    We induct on $i\geq 1$. When $i=1$ the claim reads $v_0w_{k-2},v_{k-1}w_{-1} \in E(G_1)$, which holds since $v_0$ and $w_{k-2}$, and similarly $v_{k-1}$ and $v_0$, are  
    endpoints of paths of length $k-1$ in $G$.
    Suppose that $i\geq 2$ and the above size constraints are satisfied.
    Set
        \[
        j_s := \frac{(k-1)^{i-1}-(k-1)}{2} + s\cdot (k-1)^{i-1}    \quad,\quad 0\leq s \leq (k-2)/2.
        \]
    The induction hypothesis gives 
            \[
            v_{j_0}w_{j_0+k-2},v_{j_0+k-1}w_{j_0-1}  \in E(G_{i-1})
            \]
    and Lemma \ref{lem:cycles_distance} assures that any two vertices of distance $(k-1)^{i-1}$ in $G$ are adjacent at time $i-1$.
    We have 
        \[
        j_s+k-2 \equiv k-2 \not\equiv 0 \equiv j_{s'} \mod k-1
        \]
    for any $0\leq s,s'\leq (k-2)/2$  and hence $w_{j_s+k-2} \neq v_{j_{s'}}$.
    Similarly, $w_{j_s-1} \neq v_{j_{s'}+k-1}$.
    Therefore 
        \[
        v_{j_{(k-2)/2}}\ldots v_{j_1}v_{j_0}w_{j_0+k-2}w_{j_1+k-2}\ldots w_{j_{(k-2)/2}+k-2}
        \]
    and 
        \[
        v_{j_{(k-2)/2}+k-1}\ldots v_{j_1+k-1}v_{j_0+k-1}w_{j_0-1}w_{j_1-1}\ldots w_{j_{(k-2)/2}-1}
        \]
    are paths of length $k-1$ in $G_{i-1}$.
    The claim now follows from the observation that
        \[
        j_{(k-2)/2} = \frac{(k-1)^i-(k-1)}{2}.
        \]
    \end{proof}

    \begin{clm}\label{clm:even_longwalk}
    Suppose that $\ell+\ell' > (k-1)^r-(k-1)\cdot F'(k-2,k) - 3(k-1)^2$. Then $xy\in E(G_r)$.
    \end{clm}

    \begin{proof}
    Recall that $\ell \equiv \ell' \mod 2$ since $v_\ell\ldots v_0w_0\ldots w_{\ell'}$ is an odd walk.
    Our plan is to find an $xy$-path of length $k-1$ in $G_{r-1}$. By  the conditions on $\ell,\ell'$ in Claim \ref{clm:even_shortestwalk} and   the upper bound in \eqref{eq:rF'},  we have
        \begin{linenomath}     \begin{align}
        \ell \geq \ell' 
        &\geq \ell' + \ell - (n-2) \nonumber\\
        &> (k-1)^r - (k-1)\cdot F'(k-2,k) - 3(k-1)^2 - \frac{(k-1)^r-(k-1)}{2} + F'(k-2,k) \nonumber\\
        &= \frac{(k-1)^{r}}{2} - (k-2)\cdot F'(k-2,k) - 3(k-1)^2 + \frac{k-1}2 
    \label{eq:even_lowerboundonell}
         \end{align} \end{linenomath}
    Now choose
        \[
        j_0 \in \left\lbrace \frac{(k-1)^{r-1}-(k-1)}{2}, \frac{(k-1)^{r-1}-(k-1)}{2}+k-1 \right\rbrace
        \]
    and
        \[
        j'_0 \in \left\lbrace \frac{(k-1)^{r-1}-(k-1)}{2}-1, \frac{(k-1)^{r-1}-(k-1)}{2}+k-2 \right\rbrace
        \]
    such that $\ell-j_0 \equiv \ell'-j'_0 \equiv (k-2)/2 \mod 2$.
    The congruences $\ell-j_0 \equiv \ell'-j'_0 \mod 2$ and $\ell \equiv \ell' \mod 2$ together imply $j_0 \equiv j'_0 \mod 2$.
    Thus $v_{j_0}w_{j'_0}$ is one of the edges whose presence at time $r-1$ is guaranteed by Claim \ref{clm:even_midedge} with $i=r-1$. We note here that $v_{j_0}, w_{j_0'}$ indeed exist as $\ell, \ell'\geq j_0,j_0'$ due to \eqref{eq:even_lowerboundonell} and the fact that $r$ is sufficiently large.

    We will now construct a  $v_{j_0}v_\ell$-path $Q\subseteq G_{r-1}$ and a $w_{j'_0}w_{\ell'}$-path $P \subseteq G_{r-1}$ , both of length  $(k-2)/2$  and such that $V(P)\cap V(Q)=\emptyset$. Then the union of these paths along with the edge $w_{j'_0}v_{j_0}$ gives an $xy$-path of length $k-1$ in $G_{r-1}$, and hence $xy\in E(G_r)$ as required. To this end we define 
        \[
        j'_s := j'_0 + s\cdot (k-1)^{r-1} \qquad \mbox{ and } \qquad j_s := j_0 + s\cdot (k-1)^{r-1},
        \]
    for $1\leq s\leq t:= \frac{k-2}2-1=\frac{k-4}{2}$, and claim that $Q:=v_{j_0}v_{j_1}\ldots v_{j_t}v_{\ell}$ and  $P:= w_{j'_0}w_{j'_1}\ldots w_{j'_{t}}w_{\ell'}$ are the required paths.

    First let us check that the vertices used in the path actually exist. That is, we need to check that $\ell>j_t$ and $\ell'>j'_t$. This follows because 
        \begin{linenomath} \begin{align}
        \nonumber \ell'-j_t &=\ell'-j_0-t(k-1)^{r-1} \\ 
        \nonumber    & \geq \ell'-\left(\frac{(k-1)^{r-1}-(k-1)}{2}+k-1\right)-\left(\frac{k-4}{2}\right)(k-1)^{r-1} \\ \nonumber
        &= \ell'-\left(\frac{k-3}{2}\right)(k-1)^{r-1}-\frac{k-1}{2}
        \\ \label{eq:ell-jt lower} &>(k-1)^{r-1}- (k-2)\cdot F'(k-2,k) - 3(k-1)^2>k,
        \end{align} \end{linenomath}
    where we used \eqref{eq:even_lowerboundonell} in the second last inequality and the fact that $r$ is sufficiently large in the final inequality. This shows that $\ell>j_t$ and $\ell'>j'_t$ as $\ell\geq \ell'$ and $j'_t\leq j_t+(k-2)$. 

    Next we show that the paths $P,Q$  indeed exist in $G_{r-1}$. The existence of the edges $v_{j_{s-1}}v_{j_{s}}\in E(G_{r-1})$ for $s\in[t]$ is guaranteed by Lemma~\ref{lem:cycles_distance} and likewise for the edges $w_{j'_{s-1}}w_{j'_{s}}$ with $s\in [t]$. It remains to establish that  $v_{j_t}v_\ell,w_{j'_t}w_{\ell'}\in E(G_r)$.   For this we note that 
        \begin{equation} \label{eq:ell parity}
        \ell-j_t\equiv (\ell-j_0)-(j_t-j_0)
        \equiv \left(\frac{k-2}{2}\right) - \left(\frac{k-4}{2}\right)\equiv 1 \mod 2,
        \end{equation}
    from our choice of $j_0$. Similarly, we have that $\ell'-j'_t\equiv 1 \mod 2$. Moreover, $\ell,\ell'+1\leq n-2$ (Claim \ref{clm:even_shortestwalk}) and $j_t, j'_t+1\geq \frac{k-3}{2}(k-1)^{r-1}-\frac{k-1}{2}$. Therefore, appealing to the upper bound of \eqref{eq:rF'}, we get 
        \[
        \ell-j_t,\ell'-j'_{t}< \frac{(k-1)^r-(k-1)}{2}-F'(k-2,k)-\left(\frac{k-3}{2}(k-1)^{r-1}-\frac{k-1}{2}\right)=(k-1)^{r-1}-F'(k-2,k).
        \]
    Hence Proposition \ref{prop:even_path} gives that both $v_{j_t}v_\ell$ and $w_{j'_t}w_{\ell'}$ are present in $G_{r-1}$. 

    Finally then, we need to establish that $P$ and $Q$ are disjoint.  Recall that we chose $j_0, j'_0$ such that $j_0 \equiv j'_0 \mod 2$ and hence we obtain either $j_0 = j'_0 - (k-2)$ or $j_0 = j'_0 + k$.
    Therefore
        \[
        j_s \equiv j_0 \not\equiv j'_0 \equiv j'_{s'} \mod k-1
        \]
    for $0 \leq s,s' \leq t$  and so $Q\setminus \{v_\ell\}$ and $P\setminus \{w_{\ell'}\}$ do not intersect by Remark \ref{rem:no intersect}. Moreover $v_\ell\neq w_{\ell'}$ as $\{v_\ell,w_{\ell'}\}=\{x,y\}$ and $\ell\geq \ell'>j_t$ using \eqref{eq:ell-jt lower} gives that $v_\ell\notin V(P)$ and $w_{\ell'}\notin V(Q)$. This shows that $P\cup Q\cup\{w_{j'_0}v_{j_0}\}$ is indeed a path of length $k-1$ in $G_{r-1}$ and $xy
    \in E(G_r)$ as required.
    \end{proof}

\section{Multiple cycles} \label{sec:multiple}
In this section, we prove Theorem~\ref{thm:multiple_cycles}. 
 
\subsection{Lower bound}
The lower bound of the first part is obtained by the starting graph which is disjoint union of cycles of lengths $k_2,\ldots,k_s$ and a path of length $n-(k_2+\ldots+k_s)$.
That is, let $G:= C_{k_2} \sqcup \ldots \sqcup C_{k_s} \sqcup P_{n-(k_2+\ldots+k_s)}$ with $H$-process $(G_i)_{i\geq 0}$, and let $(P^i)_{i\geq 0}$ be the $C_{k_1}$-process on $P_{n-(k_2+\ldots+k_s)}\subseteq G$.
Every copy of $C_{k_1}$ in $P^i$ can be extended to a copy of $H$ by the cycles of $G$.
Therefore $P^i \subseteq G_i$ for every $i\geq 0$, so any two vertices of odd distance on $P^0$ will eventually be adjacent in the $H$-process on $G$ by Theorem \ref{thm:strategy_upperbound} and Observation \ref{obs:hom}.
The following is an analogue of the first part of Lemma \ref{lem:cycles_distance} for multiple cycles:

    \begin{clm} \label{clm:mult lower}
    For any $x,y \in V(P^0)$ and $i\geq 0$, the distance $\dist_{G_i}(x,y)$ satisfies
        \[
       \dist_{G_0}(x,y)  \leq (k_1-1)^i \dist_{G_i}(x,y).
        \]
    \end{clm}

    \begin{proof}
    Let $Q$ be a shortest $xy$-path in $G_i$.
    Any edge $uv$ of $Q$ that is not present at time $i-1$ yields a $uv$-path $Q_{uv}$ of length $k_j-1$ in $G_{i-1}$ for some $j\in [s]$.
    We can build an $xy$-walk in $G_{i-1}$ by replacing every $uv \in E(Q)\cap E(G_i) \setminus E(G_{i-1})$ by $Q_{uv}$.
    That walk has length at most $(k_1-1)\cdot \dist_{G_i}(x,y)$ because $k_1 \geq k_j$ for $j\in [s]$.
    Therefore,
        \[
        \dist_{G_{i-1}}(x,y) \leq (k_1-1) \dist_{G_i}(x,y)
        \]
    and  iterating gives the desired claim. 
    \end{proof}
    
Let $x$ be an endpoint of $P^0$.
Let $y$ be the other endpoint if the length of $P^0$ is odd and the unique $P^0$-neighbour of the other endpoint otherwise.
With these choices $\dist_{P^0}(x,y)$ is odd and at least $n-(k_2+\ldots+k_s)-2$.
Recall that $\dist_{P^0}(x,y) = \dist_{G_0}(x,y)$.
    
Thus, for $i_0 := \ceil{\log_{k_1-1}(n-(k_2+\ldots+k_s)-2)}-1$, appealing to Claim \ref{clm:mult lower} gives
	\begin{linenomath}     \begin{align*}
	\dist_{G_i}(x,y) 
	\geq \frac{\dist_{G_0}(x,y)}{(k_1-1)^{i_0}} 
	\geq \frac{n-(k_2+\ldots+k_s)-2}{(k_1-1)^{i_0}}
	> \frac{n-(k_2+\ldots+k_s)-2}{n-(k_2+\ldots+k_s)-2}
	&= 1,
	\end{align*} \end{linenomath}
which implies that $x$ and $y$ cannot be adjacent at time $i$, hence $\tau_H(G) \geq i_0+1$.
The lower bound now follows from the simple estimate $\log_{k_1-1}(n) \leq \ceil{\log_{k_1-1}(n-(k_2+\ldots+k_s)-2)}+1$, which holds when $n$ is sufficiently large.
	
\subsection{Upper bound }
To obtain the upper bound suppose that $G$ is an arbitrary $n$-vertex graph with $\tau_H(G) = M_H(n)$ and $H$-process $(G_i)_{i\geq 0}$. We aim to show that 
\begin{equation} \label{eq:T def}\tau_H(G)=M_H(n)\leq T:=\log_{k_1-1}(n)+6s^6k_1^2.\end{equation}
At time $1$ there exists disjoint copies of $C_{k_1},\ldots, C_{k_s}$.
Fix any such copies and denote them by $C'_1,\ldots,C'_s$ where $C'_j$ has length $k_j$ for $j\in[s]$. Note that, as in the proof of Observation \ref{obs:connectivity}, the vertex sets of components in $G$ are fixed throughout the process and at no point in the process will an edge between two different connected components be added. We can therefore run our analysis on edges appearing only within connected components of $G$. 

    \begin{clm} \label{clm:sep comp}
    If $Z\subseteq V(G_1)$ is the vertex set of a component of $G_1$ that does not contain any of the $C'_j$, $j\in[s]$, then $G_i[Z] = G_{i+1}[Z]$ for $i\geq  M_{C_{k_1}}(n)+5k_1^2$.
    \end{clm}

    \begin{proof}
    For each $j\in[s]$, every copy $P'$ of $P_{k_j}$ at time $i$ with vertices in $Z$ can be extended to $C'_1\cup\ldots\cup C'_{k_j-1}\cup P'\cup C'_{k_j+1}\cup\ldots\cup C'_s$ so the endpoints of $P'$ are adjacent at time $i+1$. This implies that if no edge is added inside $Z$ at time $i$ then $\langle G\rangle_H[Z]=G_i[Z]$.
    If $G_1[Z]$ is already $C_{k_j}$-stable for every $j\in [s]$ or has at most $k_1$ vertices, it will be $C_{k_j}$-stable for every $j$ at time $k_1^2+1$. 
    Otherwise we have $|Z| \geq k_1+1$ and one can find a $k_j$-cycle in $G_2[Z]$ for some $j\in [s]$.
        
    Consider the case that $G_2[Z]$ has diameter at most $k_1-2$.
    Since $k_s = \min_{j\in [s]} k_j$ there must be a $k_s$-cycle $C$ in $G_3[Z]$.
    Pick a vertex $z_0\in V(C)$ and define
         \begin{linenomath} \begin{align*}
        X &:= \{ z\in Z : G_3[Z] \text{ has an even path from } z \text{ to } z_0 \}, \\
        Y &:= \{ z\in Z : G_3[Z] \text{ has an odd path from } z \text{ to } z_0 \}.
        \end{align*} \end{linenomath}
    As $G_3[Z]$ is connected, $Z = X \cup Y$ with $X\cap Y = \varnothing$ when $G_3[Z]$ is bipartite and $X=Y=Z$ otherwise.
    Let $x\in X$, $y\in Y$ be distinct vertices and choose $Z'\subseteq Z$ of size at least $k_1+1$ and at most $3k_1$ such that $x,y\in Z'$, $G_3[Z']$ is connected, and $C \subseteq G_3[Z']$.
    Let $i_0 := 3 + M_{C_{k_s}}(3k_1)$.
    By Lemma \ref{lem:cycles_smalldistance} with respect to the $C_{k_s}$-process on $G_3[Z']$, $xy \in E(G_{i_0}[Z'])$.
    Since $x$ and $y$ were chosen arbitrarily, $G_{i_0}[Z]$ is complete or contains a complete bipartite graph with partite sets $X$ and $Y$.
    If $G_{i_0}[Z]$ is a complete bipartite graph it will be $C_{k_j}$-stable for all $j$ after two more steps\footnote{Any permutation of the vertices inside one of the partite sets defines an automorphism of the complete bipartite graph, so the only way a bootstrap process can add new edges is by turning one of the partite sets into a clique. This can happen at most two times.}.
    If not, Lemma \ref{lem:cycles_completebipartitetoclique} tells us that $Z$ will be a clique at time $i_0+2 < 5k_1^2 $.
    
    Finally, we suppose that $G_2[Z]$ contains a copy of $P_{k_1}$.
    Let $i_1 := 2 + M_{C_{k_1}}(n)$.
    Lemma \ref{lem:cycles_smalldistance} applied to the $C_{k_1}$-process on $G_2[Z]$ guarantees that $G_{i_1}[Z]$ is either complete or contains a complete bipartite graph with at least $\floor{k_1/2}$ vertices in either part.
    Again, after two more steps we arrive at a graph that is $C_{k_j}$-stable for all $j\in[s]$.
    \end{proof}
    
Note that $M_{C_{k_1}}(n)+5k_1^2\leq T$ by Theorem \ref{thm:cycles} and so Claim \ref{clm:sep comp} implies that no edge will be added after time $T$ in a component that does not contain a $C'_j$. 

\vspace{2mm}

It remains to analyse components containing the cycles $C_j'$. Let  $V' := V(C'_1)\cup\ldots\cup V(C'_s)$ and for $j\in[s]$, let $U'_j$ be the set of vertices in $V(G)\setminus V'$ for which there exists a path to $C'_j$ in $G_1$ that does not involve any vertices from $C'_\ell$ for each $\ell\neq j $.   Let $V\subseteq V(G)$ be all vertices contained in some component of $G$ that contains one of the $C'_j$ and note that $V=V'\cup(\cup_{j}U'_j)$. Note also  that any $k_j$-cycle in $U'_j\cup V(C'_j)$  can be extended to a copy of $H$ in $G_i$ for $i\geq 1$. Likewise any $k_1$-cycle in $U'_j$ can be extended to a copy of $H$. In the remainder of the proof, for $i\geq 0$ and a vertex subset $U\subseteq V(G)$ we define 
\[N_{i}(U) := \{ v \in V(G)\setminus (U\cup V') : N_{G_i}(v) \cap U \neq \emptyset\} \] that is, the neighbourhood of $U$ in $G_i$ that is \textit{disjoint} from $U$ itself as well as from $V'$. In the case that $U=V(C'_j)$ for some $j\in[s]$, we simply write $N_i(C'_j)$.

\begin{clm} \label{clm:nbrU} For all $j\in [s]$ we have $U'_j \subseteq N_{\tau_0}(C'_j)$ for $\tau_0:=M_{C_{k_1}}(n)+3k_1^2\leq \log_{k_1-1}(n)+4k_1^2$.
\end{clm}
\begin{proof}
 If $u\in U_j'$ then there is a path $P_u$ in $G_1$ from $u$ to $C_j'$ avoiding the  $C'_{\ell}$ with $\ell\neq j$.  If $P_u$ has length at least $k_1+1$ (and hence $k_1+1$ vertices disjoint from $C_j'$),  Theorem \ref{thm:strategy_upperbound} applied with $k=k_1$ to $P_u-v$, where $v$ is the endpoint of $P_u$ on $C'_j$, gives that $u$ has distance at most 3 from $C_j'$ in $G_{\tau_0'}$ with $\tau_0':=M_{C_{k_1}}(n)+1$. Therefore at time $\tau_0'$ we can assume that every vertex in $U_j'$ is of distance at most $k_1+1$ from $C_j'$. 
By Lemma \ref{lem:cycles_smalldistance} with $k=k_j$ on $G_{\tau'_0}[U'_j\cup V(C'_j)]$, we have that indeed $U_j'\subseteq N_{\tau_0}(C_j')$, using that $\tau_0'+M_{C_{k_j}}(k_j+k_1)\leq \tau_0'+(k_j+k_1)^2/2 \leq \tau_0$.
\end{proof}

 Claim \ref{clm:nbrU} shows that all vertices in $V\setminus V'$ are  contained  in $\cup_{j}N_{\tau_0}(C'_j)$.  Our next claim tracks how the sets $N_{i}(C'_j)$ evolve in the process. 
	
	\begin{clm}\label{clm:multiple_growingsets}
	For any $j\in[s]$ and $i\geq 1$ the following hold:
		\begin{enumerate}
		\item \label{item: mult 1} $N_{i}(N_{i}(C'_j)) \subseteq N_{i+1}(C'_j)$.
		\item \label{item: mult 2} If $N_{i}(C'_j)$ is non-empty, either $G_{i+k_j^2}[N_i(C'_j)\cup V(C'_j)]$ is complete or $k_j$ is even and $G_{i+k_j^2}[N_i(C'_j)\cup V(C'_j)]$ contains a spanning complete bipartite graph with partite sets of size at least $k_j/2$.
		\item \label{item: mult 3} If $\ell \in [s]\setminus\{ j \}$ and $|N_i(C'_j)\cap N_i(C'_\ell)|\geq 3$, then $N_i(C'_\ell) \subseteq N_{i+k_1^2+1}(C'_j)$.
		\end{enumerate}
	\end{clm}
	
	\begin{proof}
	Let $j\in [s]$ be fixed.
	
	\eqref{item: mult 1} Every $u\in N_i(C'_j)$ has a neighbour on $C'_j$, so by going around $C'_j$ we can pick $x_j(u) \in V(C'_j)$ such that $u$ and $x_j(u)$ are the endpoints of a path of length $k_j-2$ in $V(C'_j)\cup \{ u \}$.
	Therefore, if $uv\in E(G_i)$ for some $v \in V(G)\setminus (N_{i}(C'_j)\cup V')$, we have $x_j(u)v\in E(G_{i+1})$ and thus $v\in N_{i+1}(C'_j)$.
	Here it is important that $v\notin V'$ so we may extend the path of length $k_j-2$ to a copy of $H$ minus an edge.
	
	\eqref{item: mult 2}
	Recall that any $k_j$-cycle with vertices in $N_{i}(C'_j)\cup V(C'_j)$ can be extended to a copy of $H$ using $C'_1,\ldots,C'_s$.
	If $j$ is odd, then for any non-empty $U\subseteq N_i(C'_j)$ of size at most two, $G_{i+k_j^2}[U\cup V(C'_j)]$ is complete, due to Lemma \ref{lem:cycles_smalldistance} and the fact that  $M_{C_{k_j}}(k_j+|U|) < \binom{k_j+2}2 \leq k_j^2+1$. Hence $G_{i+k_j^2}[N_i(C'_j)\cup V(C'_j)]$ is complete. Similarly, if $j$ is even and $C'_j$ is bipartite with parts $X\subseteq V(C'_j)$ and $Y\subseteq V(C'_j)$, then let $X'\subseteq N_i(C'_j)$ be all vertices with a neighbour in $Y$ and $Y':=N_i(C'_j)\setminus X'$, noting that each vertex in $Y'$ must have a neighbour in $X$. Similarly to the odd case, for any non-empty $U\subseteq N_i(C'_j)$ with $|U\cap X'|,|U\cap Y'|\leq 1$, we have that $G_{i+k_j^2}[U\cup V(C'_j)]$ contains a complete bipartite graph with parts $X\cup (X'\cap U)$ and $Y\cup (Y'\cap U)$, by Lemma \ref{lem:cycles_smalldistance}. Hence $G_{i+k_j^2}[N_i(C'_j)\cup V(C'_j)]$  contains a spanning complete bipartite graph with parts $X\cup X'$ and $Y\cup Y'$ (which each have size at least $|X|=|Y|=k_j/2$).

 \begin{figure}[h]
    \centering
    \includegraphics[scale=0.94]{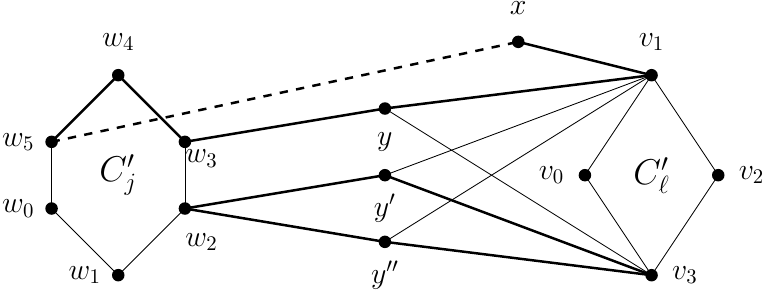}
    \caption{An example of an instance of $C'_j$ and $C'_\ell$ in \eqref{item: mult 3} with $j=6$ and $\ell=4$.}
    \label{fig:2cycles}
    \end{figure}

	\eqref{item: mult 3}  Let $x \in N_{i}(C'_\ell)\setminus N_{i}(C'_j)$.
	If $x$ has a $G_{i+k_1^2}$-neighbour in $N_{i+k_1^2}(C'_j)$, then $x\in N_{i+k_1^2+1}(C'_j)$ by part \eqref{item: mult 1}.
	Now suppose that $x$ does not have such a neighbour.
    In particular, it is not adjacent to any vertex in $N_{i}(C'_j)$.
	Since $|N_{i}(C'_j)\cap N_{i}(C'_\ell)|\geq 3$, we have that $G_{i+k_1^2}[N_{i}(C'_\ell)\cup V(C'_\ell)]$ contains both $x$ and a vertex from $N_{i}(C'_j)$ and thus cannot be complete.
	Then part \eqref{item: mult 2} forces $k_\ell$ to be even (in particular $k_\ell \geq 4$) and $G_{i+k_1^2}[N_{i}(C'_\ell)\cup V(C'_\ell)]$ to contain a spanning complete bipartite graph.
	Label the vertices of $C'_\ell$ by $v_0,\ldots,v_{k_\ell-1}$ such that $E(C'_\ell) = \{ v_0v_1,\ldots,v_{k_\ell-1}v_0 \}$ and $xv_1 \in E(G_{\tau_0})$ (see Figure \ref{fig:2cycles}).
	Let $y,y',y''$ be three distinct vertices in $N_{i}(C'_j)\cap N_{i}(C'_\ell)$.
	As $N_{G_{i+k_1^2}}(x) \cap \{ y,y',y'' \} = \emptyset$ we have $yv_t,y'v_t,y''v_t \in E(G_{i+k_1^2})$ for all odd $t\in[0,k_\ell-1]$. 
	Label the vertices of $C'_j$ by $w_0,\ldots,w_{k_j-1}$ (and relabel the vertices $y,y',y''$ if necessary) such that $yw_3 \in E(G_{i+k_1^2})$, and  $r\in\{ 1,2 \}$ such that $y'w_r,y''w_r \in E(G_{i+k_1^2})$ (in Figure \ref{fig:2cycles} we have the case $r=2$).
	Now
		\[
		xv_1yw_3\ldots w_{k_j-1}
		\]
	is a path of length $k_j-1$ that is vertex-disjoint from the $k_\ell$-cycle
		\[
		v_3\ldots v_{k_\ell-1}y'w_ry''v_3.
		\]
	Together with the cycles $C'_t$, $t\in [s]\setminus\{ j,\ell \}$ they form a copy of $H$ minus the edge $xw_{k_j-1}$.
	Therefore, $x \in N_{i+k_1^2+1}(C'_j)$.
	\end{proof}

We have that $N_i(C'_j)\subseteq N_{i+1}(C'_j)$  for all $i\geq 1$ and $j\in [s]$. Using Claim \ref{clm:multiple_growingsets}, we now show that after many steps, one of these inclusions must be strict if the process does not terminate.

\begin{clm} \label{clm:nbr1grow}
Let $\tau_1:= s^2k_1^2+3s+k_1^2+4\leq 2s^2k_1^2$. For $i\geq \tau_0$, if $\tau_H(G)>i+\tau_1$ then there is some $j\in [s]$ such that $N_{i+\tau_1}(C'_j)\neq N_i(C'_j)$. 
\end{clm}
\begin{proof}
    Suppose to the contrary that  $\tau_H(G)>i+\tau_1$ and $W_j:=N_i(C'_j)$ is such that $N_{i+\tau_1}(C'_j)=W_j$ for all $j\in [s]$. By part \eqref{item: mult 2} of Claim \ref{clm:multiple_growingsets},  for each $j\in [s]$ we have that $G_{i+k_1^2}[W_j\cup V(C'_j)]$ is either complete or contains  a spanning complete bipartite graph with partite sets of size at least $k_j/2$. Unless $G_{i+k_1^2}[W_j\cup V(C'_j)]$ is a complete bipartite graph, Lemma \ref{lem:cycles_completebipartitetoclique} then gives that $G_{i+k_1^2+2}[W_j\cup V(C'_j)]$ is complete. Now consider the steps 
    \begin{equation} \label{eq:I'} I':=\{i':i+k_1^2+2<i'\leq i+\tau_1-1\}.\end{equation} For each such step $i'\in I'$, there is an edge $e_{i'}\in E(G_{i'})\setminus E({G_{i'-1}})$  (as $\tau_H(G)>i+\tau_1$). If $e_{i'}$ has one vertex in $V'$ and the other in $V\setminus V'$ then there will be some $j\in [s]$ with $N_{i'}(C'_j)\neq N_{i'-1}(C'_j)$, a contradiction. Therefore either $e_{i'}\subseteq V'$ or $e_{i'}\cap V'=\emptyset$. There are at most $\binom{|V'|}{2}\leq \binom{sk_1}{2}\leq s^2k_1^2$ steps $i'\in I'$ for which the former can happen. If the latter occurs, as the sets $W_j$ cover $V\setminus V'$ (indeed the smaller sets $N_{\tau_0}(C'_j)$ cover $V\setminus V'$ by Claim \ref{clm:nbrU}) we have that there is some $j$ such that $e_{i'}\cap W_j\neq \emptyset$ and we claim that $e_{i'}\subseteq W_j$. If this was not the case then the other endpoint of $e_{i'}$ would lie outside of $W_j\cup V'$ and part \eqref{item: mult 1} of Claim \ref{clm:multiple_growingsets} gives that $N_{i'+1}(C'_j)\neq N_{i'}(C'_j)$, a contradiction. There can therefore be at most $3s$ steps $i'\in I'$ such that $e'$ has both endpoints in $V\setminus V'$. Indeed, as we noted above, for each $j\in [s]$, we have that $G_{i+k_1^2+2}[W_j\cup V(C'_j)]$ is either complete or complete bipartite. In the first case, there are no time steps $i'\in I'$ with $e_{i'}\subseteq W_j$ and in the second case there are at most 3, as if $i'_0$ is the first such time step then $G_{i'_0+2}[W_j\cup V(C'_j)]$ is complete by Lemma \ref{lem:cycles_completebipartitetoclique}.  Putting this all together and counting the steps $i'$ according to each possibility, we get that $|I'|\leq s^2k_1^2+3s<\tau_1-k_1^2-3$, contradicting the definition of $I'$.  
\end{proof}

Claim \ref{clm:nbr1grow} shows that after a certain amount of time steps without stabilising, one of the sets $N_i(C'_j)$ must grow. Our next claim shows that after more steps, it must in fact contain one of the other sets $N_{\tau_0}(C'_\ell)$.

\begin{clm} \label{clm:biggrow}
    Let $\tau_2:=2s^2\tau_1+k_1^2+1\leq 5s^4k_1^2$. For $i\geq \tau_0$, if $\tau_H(G)>i+\tau_2$, then there are $j,\ell\in [s]$ with $N_{\tau_0}(C'_\ell)\nsubseteq N_i(C'_j)$ and $N_{\tau_0}(C'_\ell)\subseteq N_{i+\tau_2}(C'_j)$. 
\end{clm}
\begin{proof}
    Suppose that $\tau_H(G)>i+\tau_2$. By Claim \ref{clm:nbr1grow}, for each $r\in [2s^2]$, there is some $j(r)\in [s]$ such that $|N_{i+r\tau_1}(C'_{j(r)})|>|N_{i+(r-1)\tau_1}(C'_{j(r)})|$. By averaging, there is some  $j\in [s]$ such that $|Z_j|\geq 2s$ where $Z_j:=N_{i+2s^2\tau_1}(C_j')\setminus N_{i}(C'_j)$. Now for each $z\in Z_j$, by Claim \ref{clm:nbrU}, we have that there is some $\ell\in [s]\setminus \{j\}$ such that $z\in N_{\tau_0}(C_\ell')$. By the pigeonhole principle, as $|Z_j|\geq 2(s-1)+1$, there is in fact some $\ell\in [s]\setminus \{j\}$ such that $|Z_j\cap N_{\tau_0}(C'_\ell)|\geq 3$. In particular, as $Z_j\cap N_{\tau_0}(C'_j)=\emptyset$, we have that $N_{\tau_0}(C'_\ell)\nsubseteq N_i(C'_j)$. Finally then by part \eqref{item: mult 3} of Claim \ref{clm:multiple_growingsets}, we have that $N_{\tau_0}(C_\ell')\subseteq N_{i+2s^2\tau_1}(C'_\ell)\subseteq N_{i+2s^2\tau_1+k_1^2+1}(C'_j)$ as required. 
\end{proof}

If $\tau_H(G)>\tau_0+s^2\tau_2$ then by Claim \ref{clm:biggrow} for each $t\in [s^2]$, there is some $j=j(t)$ and $\ell=\ell(t)\in [s]\setminus j(t)$ such that $N_{\tau_0}(C_{\ell}')\nsubseteq N_{i+(t-1)\tau_2}(C'_j)$ and $N_{\tau_0}(C_{\ell}')\subseteq N_{i+t\tau_2}(C'_j)$. As the sets $N_i(C'_j)$ are monotone increasing with respect to $i$, we cannot have $t<t'$ with $j(t)=j(t')$ and $\ell(t)=\ell(t')$ (as then $N_{i+(t'-1)\tau_2}(C'_j)$ would contain $N_{\tau_0}(C'_\ell)$). As there are at most $s(s-1)$ choices for a  pair $(j(t),\ell(t))$  and $s^2$ choices for $t$, we get a contradiction. 

Hence we must have \[\tau_H(G)\leq \tau_0+s^2\tau_2\leq \log_{k_1-1}(n)+4k_1^2+5s^6k_1^2\leq  \log_{k_1-1}(n)+6s^6k_1^2,\]
establishing \eqref{eq:T def} and concluding the proof of Theorem \ref{thm:multiple_cycles}.

\section*{Acknowledgements}

We are very grateful to the anonymous referee for their careful reading of the manuscript and their helpful suggestions. 

\bibliography{Biblio}

\end{document}